\documentclass[11pt]{article}
\usepackage{graphicx}
\usepackage{amsmath}%
\usepackage{amsfonts}%
\usepackage{amssymb}%
\usepackage{graphicx}
\usepackage[utf8]{inputenc} %para que reconozca las tildes y creo que también eñes
\usepackage{dsfont}
\usepackage{comment}
\usepackage{cancel}
\usepackage{epstopdf}
\usepackage{psfrag}
\usepackage{epsf}
\usepackage{changepage}
\usepackage{color}
\usepackage{marginnote}
\usepackage{hyperref}
\usepackage{cleveref}
\usepackage{multirow}
\crefformat{equation}{(#2#1#3)}

\DeclareMathOperator{\Tr}{Tr}

\DeclareMathOperator{\diag}{diag}
\newtheorem{definition}{Definition}
\newtheorem{lemma}{Lemma}
\newtheorem{theorem}{Theorem}
\newtheorem{proposition}{Proposition}
\newtheorem{corollary}{Corollary}
\newtheorem{stability}{Stability test}
\newtheorem{remark}{Remark}
\newenvironment{proof}[1][Proof]{\textbf{#1.} }{\ \rule{0.5em}{0.5em}}

\newcommand{\arrays}[2]{\begin{array}{#1} #2 \end{array}}

\newcommand{\matdos}[4]{
	\begin{pmatrix}
		#1 & #2 \\
		#3 & #4
	\end{pmatrix}
}
\renewcommand{\vec}[1]{\mathbf{#1}}
\newcommand{\vecdos}[2]{
	\begin{pmatrix}
		#1 \\
		#2 
	\end{pmatrix}
}

\newcommand{\parentesis}[1]{\left (  #1 \right ) }

\renewcommand{\d}{\operatorname{d}\!}
\newcommand{\id}{\mathds 1}

\title{The spin-spin model and the capture into the double synchronous resonance}

	\author{ Mauricio Misquero
	\thanks{Department of Mathematics, University of Rome Tor Vergata, Via della Ricerca Scientifica 1, 00133 Rome, Italy. 
	(email: misquero@mat.uniroma2.it).
}
}

\begin{document}
\maketitle

\begin{abstract}
	The aim of this article is to propose a model, that is a planar version of the Full Two-Body Problem, and discuss the existence and stability of a relevant periodic solution. Consider two homogeneous ellipsoids orbiting around each other in fixed coplanar Keplerian orbits. Moreover, their respective spin axes are assumed to be perpendicular to the orbital plane, that is also a common equatorial plane. The spin-spin model deals with the coupled rotational dynamics of both ellipsoids. For a non-zero orbital eccentricity, it has the structure of a non-autonomous system of coupled pendula. This model is a natural extension of the classical spin-orbit problem for two extended bodies. In addition, we consider dissipative tidal torques, that can trigger the capture of the system into spin-orbit and spin-spin resonances. In this paper we give some theoretical results for both the conservative model and the dissipative one. The conservative model has a Hamiltonian structure. We use properties of Hamiltonian systems to give some sufficient conditions in the space of parameters of the model, that guarantee existence, uniqueness and linear stability of an odd periodic solution. This solution represents a double synchronous resonance in the conservative regime. Such solution can be continued to the dissipative regime, where it becomes asymptotically stable. We see asymptotic stability as a dynamical mechanism for the capture into the double synchronous resonance. Finally we apply our results to several cases including the Pluto-Charon binary system and the Trojan binary asteroid 617 Patroclus, target of the LUCY mission.
\end{abstract}

	\textbf{Keywords:} Celestial mechanics, Hamiltonian systems, Dissipative systems, Rotational dynamics, Coupled oscillators, Two-Body problem.

%\tableofcontents

\section{Introduction}\label{sec:intro}%%%%%%%%%%%%%%%%%%%%%%%%%%%%%%%%%%%%%%%%%%%%%%%%%%%%%%%%%%%%%%%%%%%%%%

\subsection{Motivation}

The model we propose here is a natural extension of the well known spin-orbit problem of celestial mechanics. The spin-orbit model is an elementary, but not trivial, model to study the rotational dynamics of a satellite about its center of mass when it orbits around a planet. Here the planet acts as a point mass and the satellite is an extended body whose spin axis is perpendicular to the orbital plane. This model has the structure of a nearly integrable and periodically forced pendulum. It has attracted much attention not only for its accurate physical implications but also for its mathematical richness. Some pioneer papers are \cite{bel1966} for the conservative case and \cite{gol1966} for the dissipative case. This model contributes to explain the synchronization of the rotational motion of the Moon and its orbital motion around the Earth. In other words, the Moon is in a 1:1 spin-orbit resonance. This phenomenon is indeed very common in the solar system for natural satellites that are close enough to their respective planets, \cite{mur}. Besides, Mercury, as an orbiting body around the Sun, is locked in a 3:2 spin-orbit resonance. According to \cite{corlas2004}, in its chaotic evolution, Mercury could have reached large orbital eccentricities that made possible the capture into this higher order resonance. It is accepted that the phenomenon of capture into resonances is driven by dissipative torques, caused by internal frictions within the satellite, \cite{mac1964}. The concept of stability of a resonance in the conservative regime is linked to the concept of capture in the dissipative case and both can be related. In one hand, \cite{cel1990} studies the KAM stability in the conservative case, whereas  \cite{cel2009} proves the existence of quasiperiodic attractors for the dissipative problem, that bifurcate from the KAM tori of the conservative case. On the other hand, \cite{misort2020} proves the existence of an asymptotically stable solution in 1:1 resonance that is a continuation of a linearly stable odd periodic solution of the conservative case. The onset of chaos is another interesting feature of this problem. The oblateness of the satellite produces chaotic regions in the phase space that surround the libration regions of resonances. Chaotic zones can be very large due to overlapping of different resonances, \cite{chi1979}. A large eccentricity emphasizes this behavior, as in the case of Hyperion, \cite{wis1984}, \cite{wis1987}.

The Full Two-Body Problem (F2BP) deals with the dynamics of two extended bodies interacting gravitationally. It has been extensively investigated, especially in the last two decades, due to an increasing interest on binary systems. Due to its complexity, most of the studies are numerical explorations of particular cases, see \cite{fah2006} or \cite{comlem2014}. There are some works with a more analytical approach dealing with relative equilibria and stability, \cite{sch2002} and \cite{mac1995}. The spin-spin model is motivated mainly by \cite{boulas2009}, \cite{sch2009}, \cite{davsch2020} and \cite{batmor2015}. In one hand, \cite{boulas2009} is focused on the  evolution of the orbit and the spin axes of the bodies in the secular F2BP (averaging over fast angles). This paper points out that the mutual influence in the spin dynamics is contained in the terms of order $1/r^5$ of the expansion of the potential energy of the system, where $r$ is the distance between the bodies. On the other hand, \cite{sch2009} studies the relative equilibria and stability in the planar case, i.e., the spin axes of the bodies are perpendicular to the orbital plane, that is also a common equatorial plane. \cite{davsch2020} studies the observability of non-planar stable oscillations around the double synchronous equilibrium in binary asteroids. In \cite{sch2009} and \cite{davsch2020}, only terms up to $1/r^3$ of the potential energy are considered, so the resulting system is equivalent to two uncoupled spin-orbit problems. The planar spin-spin coupling was first studied in \cite{batmor2015}, making an analogous study as the classical paper \cite{gol1966} on the spin-orbit coupling. Particularly, \cite{batmor2015} studies the spin of the body 1, identified with two point masses slightly separated from each other (dumbbell model), that moves in a circular orbit around the body 2, an ellipsoid with uniform rotation. They focus on the case when the orbital motion is slow and the angular velocity of the body 1 becomes commensurable with the angular velocity of the body 2 (spin-spin resonance). 

The model we propose in this paper deals with the complete coupled dynamics of the F2BP in the planar and ellipsoidal case. As usual in the spin-orbit problem, we also assume that the orbital motion takes place in Keplerian ellipses. This reduces the high dimensional phase space of the F2BP to a problem of two degrees of freedom (spins) plus time-dependence (orbit). For a small non-zero orbital eccentricity, it has the structure of a nearly integrable system of coupled pendula that is periodically forced. This setting is suited to study the phenomena related to spin-orbit and spin-spin resonances. Furthermore, the intrinsic dissipative nature of the capture into resonances supports the relevance of this model. The reason is that the most used family of dissipative torques, see \cite{mac1964}, is of order $1/r^6$, whereas the spin-spin coupling appears at order $1/r^5$. In addition to the questions related to the spin-orbit problem, this model of coupled oscillators opens new questions that were not possible to consider before. We will discuss this in \Cref{sec:discussion}.

\subsection{Setting of the model.}\label{sub:model}

Consider two homogeneous ellipsoids $\mathcal E_1$ and $\mathcal E_2$ with respective masses $M_j$, $j=1,2$, principal moments of inertia $\mathcal  A_j<\mathcal B_j< \mathcal C_j$ and corresponding principal semi-axes $\mathsf a_j>\mathsf b_j>\mathsf  c_j$. Assume that the orbital motion of the ellipsoids is the same as for two point masses, say, the centers of the ellipsoids describe coplanar Keplerian orbits of eccentricity $e\in[0,1)$ with a common focus at the center of mass of the system. Moreover, assume that the spin axis of each body is the principal axis associated to $\mathsf c_j$ and is perpendicular to the orbital plane.

\begin{figure}
	\centering
	\scalebox{0.5}{\includegraphics{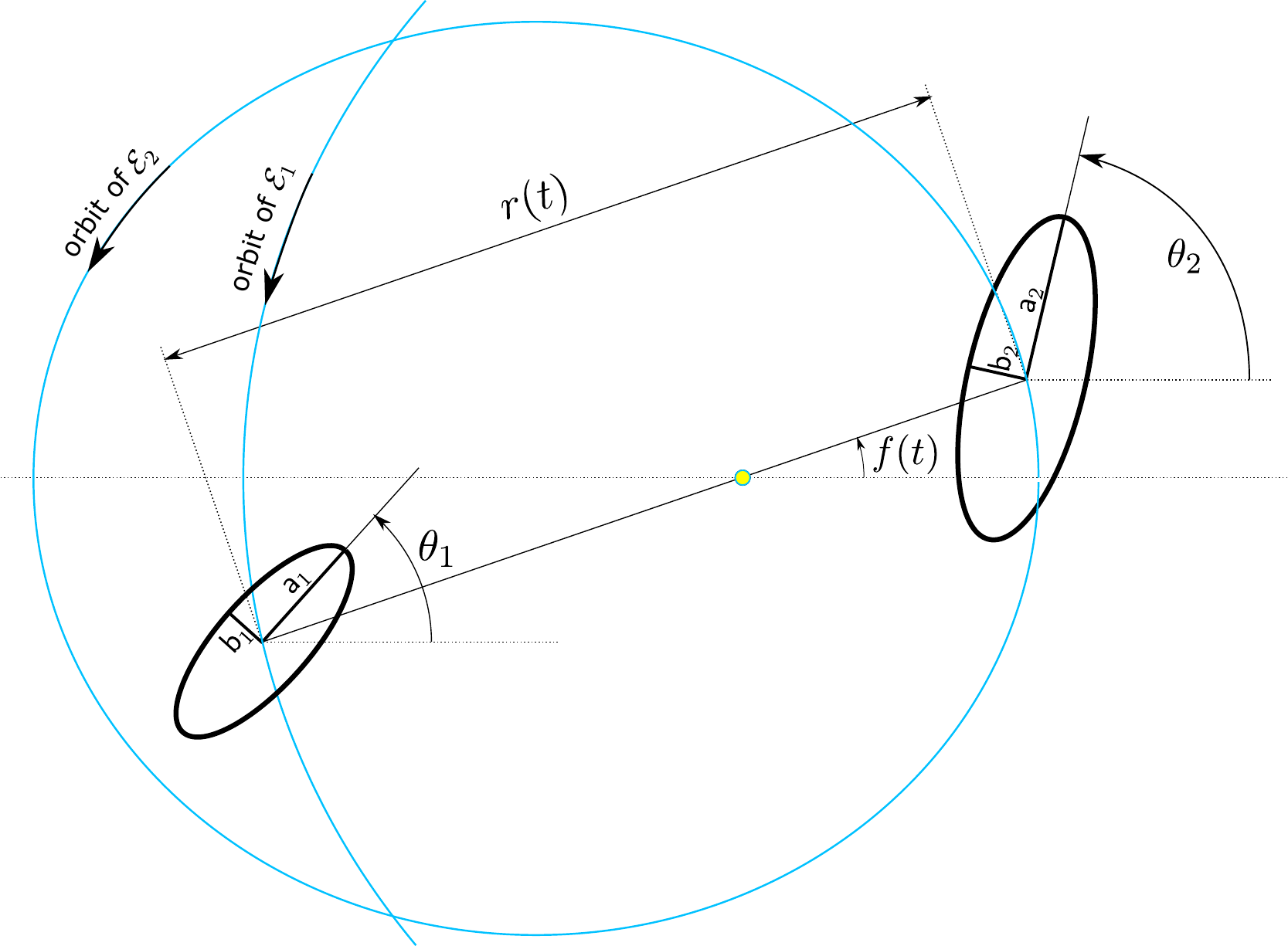}}
	\caption{The planar spin-spin problem.}
	\label{fig:spin-spin}
\end{figure}

\begin{comment}
Then, the parameter $\epsilon= \frac{3}{2} \frac{\mathcal B-\mathcal A}{\mathcal C}$ measures the oblateness of the satellite in the orbital plane.
\end{comment} 

Let us identify the orbital plane with the complex plane $\mathds C$. Consider the center of mass of the system fixed at the origin and let the center of each ellipsoid be $\vec r_j$, then, $M_1 \vec r_1+M_2 \vec r_2=0$. If we define the relative position vector $\vec r= \vec r_2-\vec r_1$ and choose the units of mass such that $M_1 +M_2 =1$, then, $\vec r_1=-M_2 \vec r$ and $\vec r_2=M_1 \vec r$. The orbital motion is defined by $\vec r$, which can be written as $\vec r=r \exp(i f)\in \mathds C$, where $r>0$ and $f$ are real functions of the time. Note that $\vec r$ describes an ellipse of eccentricity $e\in[0,1)$ and semimajor axis $a$ with focus at the origin, so, the polar coordinates $r$ and $f$ vary periodically with time, and are known by the Kepler problem. Let us take convenient units of time so that the period is $2\pi$. In the usual terminology, $f$ is called \textit{true anomaly} and the time $t$ is the \textit{mean anomaly}. There is a third useful angle $u$, the \textit{eccentric anomaly}, which is defined by the famous Kepler's equation
\begin{equation}\label{t}
t=u-e\sin u,
\end{equation}

\noindent and let us determine the Keplerian ellipse simply by
\begin{equation}\label{r}
r=a(1-e \cos u).
\end{equation}

\noindent Also, using the graphical definition of the eccentric anomaly and some geometrical relations of ellipses, we can write the position of $\vec r$ in terms of the eccentric anomaly as 
\begin{equation}\label{f}
r\exp(i f)=a(\cos u-e + i \sqrt{1-e^2} \sin u).
\end{equation}

\noindent Note that for $t=0$ we assumed that $f=u=0$, and consequently, $f=u=\pi$ when $t=\pi$. The expressions \cref{f,r} relate the true and eccentric anomalies. Moreover, \Cref{t,r,f} define $u=u(t,e)$, $\frac{r}{a}=\frac{r}{a}(t,e)$ and $f=f(t,e)$ as analytic functions in both entries. 

Recall Kepler's third law for the Two-Body Problem
\begin{equation}\label{kepler3}
G(M_1+M_2) \parentesis{\frac{T}{2\pi}}^2=a^3,
\end{equation}

\noindent where $G$ is the Gravitational constant and $T$ is the orbital period. In consequence, $G=a^3$ in our units. For our model to be completely non-dimensional and adequate to the scale of the system, we take convenient units of length such that $\mathcal  C_1+\mathcal  C_2=1$. In these units the semi-major axes $\mathsf{a}_j$ of the ellipsoids are of order $1$, whereas $a$ should be much larger. See \Cref{ap:units} for specific conversion of units.

Let $\theta_j$ be the polar angle of the principal direction associated to $\mathsf a_j$ with respect to the orbit's major axis.  See \Cref{fig:spin-spin}. The spin dynamics of the ellipsoids is modelled by the following coupled system of ordinary differential equations
\begin{equation}\label{spin_spin_dis}
\mathcal C_j \ddot \theta_j  = \mathcal T_j^\text{C}(t,\theta_1,\theta_2)+\mathcal T_j^\text{D}(t,\dot \theta_j),\qquad j=1,2,
\end{equation}

\noindent where $\mathcal T_j^\text{C}$ and $\mathcal T_j^\text{D}$ are respectively the \textit{conservative} and \textit{dissipative} torques acting on $\mathcal E_j$. 

The \textit{conservative} torque is derived from the potential gravitational energy, see \Cref{sec:model}, and it takes the form
\begin{multline}\label{T_cons}
\mathcal T_j^\text{C}(t,\theta_1,\theta_2)= - \parentesis{\frac{a}{ r(t)}}^{3} \frac{\Lambda_j}{2} \sin (2\theta_j-2f(t))   \\ 
- \parentesis{\frac{a}{ r(t)}}^{5}
\sum_{(m_1,m_2)\in \Xi }	
\frac{m_j \,\Lambda_{m_2}^{m_1}}{2}   
\sin(2 m_1(\theta_1-f(t))+2m_2 (\theta_2-f(t))),
\end{multline}

\noindent where
\begin{equation*}
\Xi=\{(m_1,m_2)\in \mathds Z^2 : \ |m_1|+ |m_2|\le 2\}.
\end{equation*}

\begin{comment}
The convergence of the series \cref*{T_cons} is guaranteed because it is a Taylor expansion of an analytic potential as we will see in \Cref{sec:model}.
\end{comment}

\noindent In \cref{T_cons} we have ignored terms of order $\parentesis{{a}/{ r(t)}}^{n}$ with $n\ge7$. The parameters $\Lambda_j$ and $\Lambda_{m_2}^{m_1}$ are positive small quantities depending on the physical parameters of the bodies and on $a$. These parameters satisfy $\Lambda_{m_2}^{m_1}=\Lambda_{-m_2}^{-m_1} < \Lambda_j<3 \mathcal C_j$. Note that if all the constants $\Lambda_{m_2}^{m_1}$ in \cref{T_cons} vanish, the system \cref{spin_spin_dis} is formed by two uncoupled spin-orbit problems in $\theta_1$ and $\theta_2$. The coupling of the system is contained in the terms $({m_1},{m_2})$ of type $(\pm 1,\pm 1)$ and $(\pm 1,\mp 1)$, whereas the rest of them are high order spin-orbit terms. 

On the other hand, the \textit{dissipative} torque $\mathcal T_j^D$ has different forms depending on the model. We will use a linear MacDonald torque \cite{mac1964}
\begin{equation}\label{macdonald}
\mathcal T_j^D(t,\dot \theta_j) = -C_{M,j} \parentesis{\frac{a}{ r(t)}}^6  \sin(2 \Delta t_j (\dot \theta_j - \dot f(t))) \approx -\delta_j \mathcal{C}_j\parentesis{\frac{a}{ r(t)}}^6 (\dot \theta_j - \dot f(t)),
\end{equation}

\noindent where $C_{M,j}$ are constants depending on the parameters of the bodies. Here we assumed that $|\Delta t_j (\dot \theta_j - \dot f(t))|\ll 1$ because the parameters $\Delta t_j$ and $\delta_j$ are very small positive numbers. This type of torque has been extensively used, taking as reference \cite{gol1966} or \cite{mur}, for example. According to \cite{efr2013}, to obtain \cref{macdonald}, the dissipation is modelled by assuming that there is a time delay between the deforming disturbance and the actual deformation of each body. That delay is a small fixed amount $\Delta t_j$ (time lag), leading to an angular lag of $ (\dot f(t,e) -\dot \theta_j )\Delta t_j$ (geometric lag). It is worth mentioning that there is no physical reason for both lags (or both $\delta_j$) to match.

Note that if $\mathcal T_j^\text{D}=0$, the system \cref{spin_spin_dis} has a Hamiltonian structure. The corresponding Hamiltonian has two degrees of freedom and time dependence and it is given by
\begin{equation}\label{ham}
H(\theta_1,\theta_2,p_{\theta_1},p_{\theta_2},t) = \frac{p_{\theta_1}^2 }{2 \mathcal C_1}+ \frac{p_{\theta_2}^2 }{2 \mathcal C_2} + \mathcal V(t, \theta_1,\theta_2),
\end{equation}
\noindent where
\begin{multline*}
\mathcal V(t, \theta_1,\theta_2) = -\frac{1}{4} \parentesis{\frac{a}{ r(t)}}^3\sum_{j=1}^2 \Lambda_j\cos (2\theta_j-2f(t)) \\-\frac{1}{4}\parentesis{\frac{a}{ r(t)}}^{5}
\sum_{(m_1,m_2)\in \Xi} 
\Lambda_{m_2}^{m_1}
\cos(2 m_1(\theta_1-f(t))+2m_2 (\theta_2-f(t))).
\end{multline*}

\noindent Due to the explicit time dependence of the Hamiltonian, the energy of the system is not constant even though $\mathcal T_j^\text{D}\equiv0$. However, if $\mathcal T_j^\text{D}\equiv0$, the system \cref{spin_spin_dis} will be called \textit{conservative}, because no dissipative forces are involved in the physical derivation of the model. On the other hand, if $\mathcal T_j^\text{D}$ is not identically zero for all time, then we will call it \textit{dissipative}. The italic font will remark this point. In \Cref{sec:model} we will see also a purely conservative version of the model involving $(r,f,\theta_1,\theta_2)$ as unknown functions of time.

There are solutions of \cref{spin_spin_dis} that are especially relevant. Since the spin-orbit problem is a particular case of \cref{spin_spin_dis}, a solution satisfying $\theta_1(t+2\pi n_o)=\theta_1(t)+2\pi n_s$, with $n_s,n_o\in \mathds Z$, is called $n_s:n_o$ spin-orbit resonance of the ellipsoid $\mathcal E_1$. The same is true for $\mathcal E_2$. Spin-spin resonances arise when the spin rates of the two ellipsoids become commensurable. In \cite{batmor2015} these resonances were studied independently from the orbital rate. There are some solutions in which the ellipsoids are simultaneously in a spin-orbit and a spin-spin resonance. The simplest of these resonances is the double synchronous resonance of equation \cref{spin_spin_dis}, that is, solutions satisfying $\theta_j(t+2\pi )=\theta_j(t)+2\pi $, for both $j=1,2$. In other words, the spin of both ellipsoids synchronize with the orbital motion at the same time.

\subsection{Setting of our approach and results}

We are going to deal with the capture into the double synchronous resonance. In the same way as in \cite{misort2020}, in this paper we will approach this phenomenon from an analytical point of view. We will look for conditions resulting in the existence of a double synchronous solution of the \textit{conservative} model that can be continued to an asymptotically stable solution of the \textit{dissipative} model. In this context, the asymptotic stability of the solution represents the phenomenon of capture into the resonance: solutions in the vicinity of the asymptotically stable solution get closer and closer to it as $t\rightarrow +\infty$.

Let us take the change of variable $\Theta_j =2(\theta_j - f)$, such that the system \cref{spin_spin_dis} turns into
\begin{multline}\label{spin_spin_dis_mayus}
\mathcal C_j\ddot \Theta_j + \delta_j \mathcal C_j \parentesis{\frac{a}{ r(t)}}^{6} \dot \Theta_j 
+\parentesis{\frac{a}{ r(t)}}^{3} \Lambda_j\sin \Theta_j \\
+\parentesis{\frac{a}{ r(t)}}^{5}
\sum_{(m_1,m_2)\in \Xi }m_j \,\Lambda_{m_2}^{m_1}	 
\sin(m_1\Theta_1+m_2 \Theta_2) = -2\mathcal C_j \ddot f(t)  .
\end{multline}

\noindent The system \cref{spin_spin_dis_mayus} models a couple of damped and forced pendula of variable length. Since $f(t+2\pi)=f (t)+2\pi$, then, double synchronous resonances correspond to solutions of \cref{spin_spin_dis_mayus} satisfying $\Theta_j(t+2\pi)=\Theta_j (t)$.

In \Cref{sec:model} we will make the derivation of the \textit{conservative} model from the Lagrangian of the physical system and obtain the expression of $\Lambda_j $ and $\Lambda_{m_2}^{m_1} $ in terms of physical parameters. In \Cref{sec:lin_stab}, we will deal with the \textit{conservative} system, say, \cref{spin_spin_dis_mayus} with $\delta_j=0$,
\begin{equation}\label{spin_spin_mayus}
\mathcal C_j\ddot \Theta_j  
+\parentesis{\frac{a}{ r(t)}}^{3} \Lambda_j\sin \Theta_j 
+\parentesis{\frac{a}{ r(t)}}^{5}
\sum_{(m_1,m_2)\in \Xi }m_j \,\Lambda_{m_2}^{m_1}	 
\sin(m_1\Theta_1+m_2 \Theta_2) = -2\mathcal C_j \ddot f(t)  ,
\end{equation}

\noindent and discuss the existence, uniqueness and linear stability of an odd $2\pi$-periodic solution. This solution is a continuation of the trivial solution $\Theta(t)\equiv 0$ for $e=0$. This will lead us to a region of linear stability in the space of parameters of the system. In this section we will use some properties of symmetric matrices and linear Hamiltonian systems with periodic coefficients. We are interested in the linear stability of the periodic solution of \cref{spin_spin_mayus} because it will allow us to find, by continuation, an asymptotically stable periodic solution for the \textit{dissipative} case \cref{spin_spin_dis_mayus} for $\delta_j>0$. This will be proved in \Cref{sec:asymp}, provided that $|\Lambda_{m_2}^{m_1}|$ and $|\delta_j|$ are small enough. In \Cref{sec:app} we will explain how to apply our results to real cases and use the Pluto-Charon system and the binary asteroid 617 Patroclus as two representative examples. We will also compare our estimates with some numerical experiments and with the spin-orbit problem. Finally in \Cref{sec:discussion} we will make a discussion about the model and our results.

\section{Derivation of the conservative spin-spin model}\label{sec:model}%%%%%%%%%%%%%%%%%%%%%%%%%%%%%%%%%%%%%%%

In this section we will compute the equations of motion of the ellipsoids with respect to the inertial frame with origin at the barycenter of the system. \Cref{sub:Lag} is devoted to find the equations of motion of the full system of four variables $( r,f,\theta_1,\theta_2)$, in terms of the gravitational potential energy $V=V(r,f,\theta_1,\theta_2)$. In \Cref{sub:kep} we fix the Keplerian orbit and obtain the final model in terms of physical parameters of the system.

\subsection{The planar Lagrangian model}\label{sub:Lag}

Let the Lagrangian of the system be $L=T-V$, where $T$ is the kinetic energy and $V$ the potential energy of the system. Recall that the positions of the bodies are $\vec r_1=-M_2 \vec r$ and $\vec r_2=M_1 \vec r$, where the relative position vector is defined by $\vec r= \vec r_2-\vec r_1=r \exp(if)$. Besides, for each body, the angle $\theta_j$ defines the orientation of the axis associated to $\mathsf a_j$. We are going to use $r$, $f$, $\theta_1$ and $\theta_2$, depicted in \Cref{fig:spin-spin}, as the Lagrangian variables of our system. The total orbital kinetic energy is given by
\begin{equation*}
T_{orb}=\frac{1}{2} (M_1 \dot{\vec r_1}^2 +M_2 \dot{\vec r_2}^2) = \frac{\mu}{2} \dot{\vec r}^2 = \frac{\mu}{2} (\dot r^2 + r^2 \dot f^2),
\end{equation*}

\noindent where $\mu=M_1M_2$ is the the reduced mass of the system (recall $M_1+M_2=1$). While the rotational kinetic energy is $T_{rot}=\frac{1}{2} \mathcal C_1 \dot \theta_1^2 + \frac{1}{2} \mathcal C_2 \dot \theta_2^2$. See \Cref{app:potential} for the derivation of the full expression of the potential energy of the system $V=V(r,f,\theta_1,\theta_2)$, equation \cref{V_full}. The Euler-Lagrange equations corresponding to the Lagrangian $L=T_{orb}(r,\dot r,\dot f) + T_{rot}(\dot \theta_1,\dot \theta_2)-V(r,f,\theta_1,\theta_2)$ are
\begin{equation}\label{spin-spin_V}
 \mathcal C_1\ddot \theta_1  = -\partial_{\theta_1} V ,\quad  \mathcal C_2\ddot \theta_2  = -\partial_{\theta_2} V,
\end{equation}
\begin{equation}\label{orbital_eqs}
\mu \ddot r = \mu r\dot f^2 - \partial_r V,\quad  \ddot f = -\frac{1}{\mu  r^2}\partial_{f} V  -2 \frac{\dot r \dot f}{r}. 
\end{equation}

In \Cref{sub:V_planar} we give the expansion of the potential energy. In the case of ellipsoids, it has the form $V=\sum_{n=0}^\infty V_{2n}$, where $V_{2n}$ is proportional to $1/r^{2n+1}$. The first terms of the expansion are
\begin{equation*} 
V_0=-\frac{GM_1M_2}{r}.
\end{equation*}
\begin{equation*}
V_2= -\frac{GM_2}{4r^3} \parentesis{q_1+3d_1\cos(2(\theta_1-f))}
-\frac{GM_1}{4r^3} \parentesis{q_2+3d_2\cos(2(\theta_2-f))},
\end{equation*}
\begin{equation}\label{V_4_ph_parameters}
\arrays{rl}{
	\displaystyle V_4=-\frac{3G}{4^3 r^5}  \{
	& 12q_1q_2 + \frac{15}{7} [\frac{M_2}{M_1}d_1^2+2\frac{M_2}{M_1}q_1^2 + \frac{M_1}{M_2}d_2^2+2\frac{M_1}{M_2}q_2^2]  \\
	&+d_1 M_2\left \{ [20\frac{q_2}{M_2} + \frac{100}{7}\frac{q_1}{M_1}]\cos (2(\theta_1-f)) +25 \frac{d_1}{M_1}\cos (4(\theta_1-f)) \right \}\\
	&+d_2 M_1 \left \{[20\frac{q_1}{M_1} + \frac{100}{7}\frac{q_2}{M_2}]\cos (2(\theta_2-f)) + 25 \frac{d_2}{M_2}\cos (4(\theta_2-f)) \right \}\\
	&+6d_1 d_2 \cos (2(\theta_1-\theta_2))+70d_1 d_2  \cos (2(\theta_1+\theta_2)-4f)
	\},
}
\end{equation}

\noindent where we defined the parameters
\begin{equation}\label{q_d}
d_j=\mathcal B_j- \mathcal A_j, \qquad 	
q_j=2\mathcal C_j- \mathcal B_j- \mathcal A_j.
\end{equation}

\noindent Note that $d_j$ is proportional to $C^{j)}_{22}$, whereas $q_j$ is proportional to $C^{j)}_{20}$, where $C^{j)}_{nm}$ are the usual coefficients in the expansion of the gravitational potential of the ellipsoid $\mathcal E_j$. The quantity $d_j/\mathcal C_j$ measures the oblateness of the section of the ellipsoid in the plane of motion, whereas, $q_j/\mathcal C_j$ measures the flattening with respect to the plane. If $\mathcal  A_j \le\mathcal B_j\le \mathcal C_j$, then, $q_j\ge d_j\ge 0$. Note that the term $V_0$ contains the dynamics of two point masses, $V_2$ the uncoupled spin-orbit dynamics and $V_4$ the spin-spin coupled dynamics between $\theta_1$ and $\theta_2$. The coupling terms appear in the last line of \cref{V_4_ph_parameters}.

\subsection{The Keplerian assumption and the spin-spin model}\label{sub:kep}

The complete dynamics of the system is given by \Cref{spin-spin_V,orbital_eqs}, with $V$ in \cref{V_full}. In this paper we impose that the orbital motion is Keplerian, i.e., we keep only $V_0$ in the orbital part \cref{orbital_eqs}. Besides, in the spin part \cref{spin-spin_V}, we truncate $V$ ignoring terms of order $1/r^7$ and higher, then $V\approx V_0+V_2+V_4$. The resulting system is
\begin{equation}\label{spin-sin_V_trun}
\mathcal C_1\ddot \theta_1  = -\partial_{\theta_1} (V_0+V_2+V_4) ,\quad  \mathcal C_2\ddot \theta_2  = -\partial_{\theta_2} (V_0+V_2+V_4),
\end{equation}
\begin{equation}\label{orbital_eqs_kep}
\mu \ddot r = \mu r\dot f^2 - \partial_r V_0,\quad  \ddot f = -\frac{1}{\mu  r^2}\partial_{f} V_0  -2 \frac{\dot r \dot f}{r}.
\end{equation}

\noindent Note that, since $\partial_{\theta_j} V_0=0$, the system \cref{orbital_eqs_kep} is now decoupled from \cref{spin-sin_V_trun}. Its solution is $r=r(t)$, $f=f(t)$ given by \Cref{t,r,f} and depends on the eccentricity of the orbit $e$ and its semi-major axis $a$.

Let us now write $V_2$ and $V_4$ in a more convenient way. The quantity $M_j a^2$ is a sort of orbital moment of inertia of the body $\mathcal E_j$. Then, we can define 
\begin{equation}\label{hat_d_q}
\hat d_j=\frac{d_j}{M_j a^2},\quad \hat q_j=\frac{q_j}{M_j a^2},
\end{equation}

\noindent so that $\hat d_j$ measures the equatorial oblateness of $\mathcal E_j$ with respect to the size of the orbit and $\hat q_j$ measures the flattening of $\mathcal E_j$ with respect to the size of the orbit.

Taking into account that in our units $G=a^3$, the terms $V_2$ and $V_4$ can be written in a compact way as
\begin{equation}\label{V_2}
V_2=-\frac{1}{4}\parentesis{\frac{a}{ r(t)}}^{3}  (\Lambda_0+\Lambda_1\cos (2\theta_1-2f(t))+\Lambda_2\cos (2\theta_2-2f(t)))
\end{equation}

\noindent and
\begin{equation}\label{V_4}
V_4=-\frac{1}{4}\parentesis{\frac{a}{ r(t)}}^{5}
\sum_{(m_1,m_2)\in \Xi} 
\Lambda_{m_2}^{m_1}
\cos(2 m_1(\theta_1-f(t))+2m_2 (\theta_2-f(t)))
\end{equation}

\noindent where
\begin{equation*}
\Xi=\{(m_1,m_2)\in \mathds Z^2 : \ |m_1|+|m_2|\le 2 \},
\end{equation*}

\noindent and the following $\Lambda$ parameters are defined by
\begin{equation}\label{Lambda_j}
\Lambda_1=3 d_1M_2,\qquad \Lambda_2=3d_2M_1,
\end{equation}
\begin{equation}\label{Lambda_1_0}
\Lambda_0^1=\Lambda_0^{-1}=\frac{5}{56}(7\hat q_2+5\hat q_1) \Lambda_1, \qquad \Lambda_1^0=\Lambda_{-1}^0=\frac{5}{56}(7\hat q_1+5\hat q_2) \Lambda_2,
\end{equation}
\begin{equation}\label{Lambda_2_0}
\Lambda_0^2=\Lambda_0^{-2}=\frac{25}{32}\hat d_1 \Lambda_1, \qquad \Lambda_2^0=\Lambda_{-2}^0=\frac{25}{32}\hat d_2 \Lambda_2,
\end{equation}
\begin{equation}\label{Lambda_1_1}
\Lambda_1^1=\Lambda_{-1}^{-1} = \frac{35}{16} \hat d_1 \Lambda_2 = \frac{35}{16} \hat d_2 \Lambda_1,\qquad
\Lambda_1^{-1}=\Lambda_{-1}^1 = \frac{3}{16} \hat d_1 \Lambda_2 = \frac{3}{16} \hat d_2 \Lambda_1
\end{equation}
\begin{equation*}
\Lambda_0= q_1M_2+ q_2M_1,\qquad\Lambda_0^0= \frac{9}{4}\hat q_1 q_2 M_1 + \frac{15}{112}(\Lambda_1 \hat d_1+6\hat q_1 q_1M_2 + \Lambda_2 \hat d_2+6\hat q_2 q_2M_1). 
\end{equation*}

\noindent With the last definitions we can write equations \cref{spin-sin_V_trun} as $\mathcal C_j \ddot \theta = \mathcal T_j^C$, where $\mathcal T_j^C=-\partial_{\theta_j} (V_2+V_4) $ are the \textit{conservative} torques of the spin-spin model shown in \cref{T_cons}. This can be checked with the expressions \cref{V_2} and \cref{V_4}. Note that $m_j\Lambda_{m_2}^{m_1}$ is in all cases proportional to the corresponding $\Lambda_j$. Then, the equations of the \textit{conservative} spin-spin model \cref{spin-sin_V_trun} can be written in terms of the physical parameters in the following symmetric way for $j=1,2$,
\begin{multline}\label{spin_spin_j}
0=   \ddot \theta_j +
  \frac{\lambda_j}{2} \Big \{   \parentesis{\frac{a}{ r(t)}}^{3} \sin(2 \theta_j-2f(t)) + \\
+  \parentesis{\frac{a}{ r(t)}}^{5} \Big [ 
\frac{5}{4} \parentesis{ \hat q_{3-j}+ \frac{5}{7}  \hat q_j} \sin(2 \theta_j-2f(t))
+\frac{25   \hat d_j}{8}  \sin(4 \theta_j-4f(t))  \\
+ \frac{3   \hat d_{3-j}}{8}  \sin(2 \theta_j-2\theta_{3-j})
+ \frac{35   \hat d_{3-j}}{8}  \sin(2 \theta_{3-j}+2\theta_j-4f(t))
\Big ] \Big\},
\end{multline}

\noindent where
\begin{equation*}
\lambda_j=\frac{\Lambda_j}{\mathcal C_j} = 3\frac{d_j}{\mathcal C_j}\frac{\mu}{M_j}.
\end{equation*}

It is worth mentioning that the terms with $\hat q_j$ and $\hat d_j$ in \cref{spin_spin_j} were missing in the model used in \cite{batmor2015} due to the dumbbell simplification for one of the bodies in the derivation of the equations. Not all the parameters appearing in \cref{spin_spin_j} are free because the following identities hold
\begin{equation}
\mathcal C_1+\mathcal C_2=1, \quad  \Lambda_1 \hat d_2=\Lambda_2 \hat d_1,\quad \Lambda_1 \hat q_2=\Lambda_2 \hat q_1.
\end{equation}

\noindent In consequence, our model depends on six independent parameters with physical meaning $(e;\mathcal C_1, \lambda_1, \lambda_2, \hat d_1,\hat q_1)$. Moreover, in \cref{spin_spin_j} we see that spin of the ellipsoid $\mathcal E_2$ is affected by the spin-spin coupling with a strength essentially given by $\hat d_1$, and vice versa.

\section{Linear stability of the double synchronous resonance in the conservative model}\label{sec:lin_stab}%%%%%%%%%%%%%%%%%%%%%%%%%%%%%%%%%%%%%%%

In this section we deal with the \textit{conservative} system with the notation in \cref{spin_spin_mayus}, that is more convenient for our purpose. The main result is \Cref{th_estabilidad}. It determines a region of linear stability of the double synchronous resonance in the space of parameters of the system.

\subsection{Existence of the odd $2\pi$-periodic solution}

The system \cref{spin_spin_mayus} can be written as
\begin{equation}\label{F}
\mathcal C \ddot \Theta  +F(t,\Theta)=0,
\end{equation}

\noindent where
\begin{equation*} \Theta=\vecdos{\Theta_1}{\Theta_2},	\quad \mathcal C=\matdos{\mathcal C_1}{0}{0}{\mathcal C_2},\quad \mathcal C_j>0,
\end{equation*}

\noindent and $F(t,\Theta)$ is the bounded function given by
\begin{multline}\label{F_Theta}
F(t,\Theta) =
\parentesis{\frac{a}{ r(t)}}^{3}  \vecdos{\Lambda_1\sin \Theta_1}{\Lambda_2\sin \Theta_2} + \\
+\parentesis{\frac{a}{ r(t)}}^{5}
\sum_{(m_1,m_2)\in \Xi } \vecdos{m_1}{m_2} \,\Lambda_{m_2}^{m_1}	 
\sin(m_1\Theta_1+m_2 \Theta_2)
	+2\ddot f(t) \vecdos{\mathcal C_1}{\mathcal C_2}.
\end{multline}

Note that equation \cref{F} is invariant under the change $(t,\Theta)\rightarrow (-t,-\Theta)$, since  $f(-t)=-f(t)$ and $r(-t)=r(t)$. Then, if $\Theta(t)$ is a solution of \cref{F}, so it is $-\Theta(-t)$. On the other hand, for $e=0$, we have $f(t)=t$ and $r(t)=a$, meaning that the system \cref{F} is that of two coupled free pendula. For this case, the trivial solution $\Theta(t)\equiv 0$ is a stable equilibrium. Then, for $e\ne0$, it is natural to look for the $2\pi$-periodic continuation of $\Theta(t)\equiv 0$  in the family of the odd solutions of \cref{F}, say, solutions satisfying $\Theta(-t)=-\Theta(t)$. This is equivalent to solve the Dirichlet problem
\begin{equation}\label{Dirichlet_e}
\left \{
\arrays{l}{
	\mathcal C \ddot \Theta  +F(t,\Theta)=0, \\
	\Theta(0)=\Theta(\pi)=0.
}
\right.
\end{equation}

It is well known from nonlinear analysis that the system \cref{Dirichlet_e} has at least one solution because $F(t,\Theta)$ is bounded. We can give a simple proof for this. Let $\Theta(t)=\vartheta(t,v)$ be the solution of \cref{F} satisfying initial conditions $\Theta(0)=0$, $\dot\Theta(0)=v\in \mathds{R}^2$. Solutions of the problem \cref{Dirichlet_e} are in correspondence with the solutions of the equation $\vartheta(\pi,v)=0$. From \cref{F}, we know that $\vartheta$ satisfies the following integral equation
\begin{equation}\label{integral_eq}
\vartheta(t,v) = v t - \int_0^t (t-s)\mathcal C^{-1}F(s,\vartheta(s,v))\d s.
\end{equation}

Let $||\cdot||$ be a norm in $\mathds R^2$, for instance, the maximum norm or the Euclidean one. We will employ the same notation for the corresponding induced matrix norm in $\mathds R^{2\times 2}$. Since there exists a positive number $M\ge ||\mathcal C^{-1}F(t,\Theta)||$, then
\begin{equation*}
||\vartheta(t,v) - v t || \le M\frac{t^2}{2},
\end{equation*}

\noindent for each $t\in \mathds R$. If we take $t=\pi$, then, $||\Phi(v)||\le M \pi/2$, with $\Phi(v)=v-\vartheta(\pi,v)/\pi$ and $v\in \mathds R^2$. Hence, we can apply Brouwer's fixed-point theorem to guarantee that $\Phi(v)$ has a fixed point for some $v_0$ satisfying $||v_0||\le M \pi/2$. For such point we have that $\vartheta(\pi,v_0)=0$, and the corresponding $\vartheta(t,v_0)$ satisfyies \cref{Dirichlet_e}.

\subsection{Uniqueness of the solution}

We know now that the Dirichlet problem \cref{Dirichlet_e} has a solution, however, it is not necessarily unique. For instance, if $\Lambda > 1$, there is not a unique solution for the free pendulum equation $\ddot x+\Lambda \sin x =0$, $x\in \mathds R$, with Dirichlet conditions $x(0)=x(\pi)=0$. See \cite{maw2004}. We would like to determine sufficient conditions on the space of parameters of the system such that there is uniqueness for the problem \cref{Dirichlet_e}.

We can prove uniqueness by a contradiction argument. Define the following matrix
\begin{equation*}
\mathcal{C}^{1/2} = \matdos{\sqrt{\mathcal C_1}}{0}{0}{\sqrt{\mathcal C_2}}
\end{equation*}

\noindent and its inverse $\mathcal{C}^{-1/2}=(\mathcal{C}^{1/2})^{-1}$. Let $\Theta^{(0)}(t)$ and $\Theta^{(1)}(t)$ be two non-identical solutions of \cref{Dirichlet_e}. Then, we can check that $y(t)=\mathcal{C}^{1/2}(\Theta^{(1)}(t)-\Theta^{(0)}(t))$ is a solution of the Dirichlet problem
\begin{equation}\label{y}
\left \{
\arrays{l}{
	\ddot y + A(t) y=0, \\
	y(0)=y(\pi)=0,
}
\right.
\end{equation}

\noindent with $A(t)$ a symmetric\footnote{In this paper we use properties of linear systems with symmetric coefficient matrices. $\mathcal C^{-1}\partial_\Theta F(t,\Theta)$ is not symmetric, but we obtain the desired structure using $\mathcal C^{1/2}$. See \cite{yak1975}.} matrix given by 
\begin{equation}\label{A_int}
\mathcal{C}^{1/2}A(t)\mathcal{C}^{1/2}=\int_0^1 \partial_\Theta F(t, \Theta^{(\lambda)}(t)) \d \lambda,
\end{equation}

\noindent where $\Theta^{(\lambda)}(t)=\lambda \Theta^{(1)}(t) + (1-\lambda) \Theta^{(0)}(t)$ and
\begin{equation*}
F(t,\Theta) = \vecdos{F_1(t,\Theta)}{F_2(t,\Theta)}, \quad 
\partial_\Theta F(t,\Theta)
=\matdos{\frac{\partial F_1}{\partial{\Theta_1}} }{\frac{\partial F_1}{\partial{\Theta_2}} }{\frac{\partial F_2}{\partial{\Theta_1}} }{\frac{\partial F_2}{\partial{\Theta_2}} }.
\end{equation*}

The statement of uniqueness is given in \Cref{propo_unicidad}. We can prove it by guaranteeing that \cref{y} has only the trivial solution. In the proof we are going to apply the following lemma to \cref{y} for a generic matrix $A(t)\in \mathds R^{d\times d}$. But first we need some definitions. Let $\langle \cdot, \cdot \rangle$ be the Euclidean inner product in $\mathds R^d$ and $||\cdot||$ its corresponding norm. Let $\id$ be the unit matrix in $R^{d\times d}$.

\begin{definition}\label{def:sym}
	Let $A_1,A_2\in \mathds R^{d\times d}$ be two symmetric matrices. We say that $A_1\le A_2$ if, for the corresponding quadratic forms, $\langle A_1 \, y , y  \rangle \le \langle A_2 \, y , y  \rangle$ for all $y\in \mathds R^d$.
\end{definition}

\begin{lemma}\label{lema_A}
	Assume that, for some $\gamma<1$, the matrix $A(t)\in \mathds R^{d\times d}$ is such that $A(t) \le \gamma \id $ for each $t\in [0,\pi]$. Then, the only solution of $\ddot y + A(t) y=0$, $y\in \mathds R^d$, with Dirichlet conditions $
	y(0)=y(\pi)=0$ is the trivial one.
\end{lemma}

\begin{proof}
	Proceed by contradiction. Let $y(t)$ be a non-trivial solution of $\ddot y + A(t) y=0$,  $
	y(0)=y(\pi)=0$, then,
	\begin{equation*}
	\int_0^\pi \langle \ddot y (t), y (t) \rangle + \int_0^\pi \langle A(t) y (t), y (t) \rangle =0,
	\end{equation*}
	
	\noindent integrating by parts it follows that
	\begin{equation*}
	\int_0^\pi ||\dot y (t)||^2 = \int_0^\pi \langle A(t) y (t), y (t) \rangle .
	\end{equation*}
	
	Let $y_n(t)$ be the components of the vector $y(t)$. From the Sobolev inequality $\int_0^\pi|y_n(t)|^2\le \int_0^\pi|\dot y_n(t)|^2$, see \cite{zha} or \cite{misort2020}, we get that
	\begin{equation*}
	\int_0^\pi ||y (t)||^2 \le  \int_0^\pi \langle A(t) y (t), y (t) \rangle.
	\end{equation*}
	
	\noindent This contradicts the hypothesis $A(t) \le \gamma \id $ for some $\gamma<1$. Then, $y (t)$ must be the trivial solution.	
\end{proof}

Let us define the matrix $\tilde A(t,\Theta)=\mathcal{C}^{-1/2}  \partial_\Theta F(t,\Theta)\mathcal{C}^{-1/2}$,
\begin{multline}\label{A_t_Th}
\tilde  A(t,\Theta)=\parentesis{\frac{a}{ r(t)}}^{3} 
\matdos{\frac{\Lambda_1}{\mathcal C_1} \cos \Theta_1}{0}{0}{\frac{\Lambda_2}{\mathcal C_2} \cos \Theta_2}
\\
+\parentesis{\frac{a}{ r(t)}}^{5}
\sum_{(m_1,m_2)\in \Xi } \matdos{\frac{m_1^2}{\mathcal C_1}}{\frac{m_1m_2}{\sqrt{\mathcal C_1\mathcal C_2}}}{\frac{m_1m_2}{\sqrt{\mathcal C_1\mathcal C_2}}}{\frac{m_2^2}{\mathcal C_2}} \,\Lambda_{m_2}^{m_1}	 
\cos(m_1\Theta_1+m_2 \Theta_2).
\end{multline}

We will use the maximum norm
\begin{equation*}
||y||=\max \{ |y_1|,|y_2| \},\quad y=\vecdos{y_1}{y_2},
\end{equation*}

\noindent and its induced norm in matrices
\begin{equation*}
||A||=\max \left \{|A_{11}|+|A_{12}|,|A_{21}|+|A_{22}| \right \},\quad A=\matdos{A_{11}}{A_{12}}{A_{21}}{A_{22}}.
\end{equation*}

\begin{theorem}\label{propo_unicidad}
	Assume that $e\in[0,1)$ and the parameters of the problem satisfy 
	\begin{equation}\label{cond_un}
	1>\frac{1}{(1-e)^3}\max \left \{ \frac{\Lambda_1}{\mathcal C_1}  (1 + \alpha_1) , \frac{\Lambda_2}{\mathcal C_2}  (1 + \alpha_2) \right \},
	\end{equation}
	
	\noindent where
	\begin{equation}\label{alpha}
	\alpha_j \frac{\Lambda_j}{\mathcal C_j} = \frac{1}{(1-e)^2} \sum_{(m_1,m_2)\in \Xi }\parentesis{\frac{m_j^2}{\mathcal C_j}+\frac{|m_1m_2|}{\sqrt{\mathcal C_1\mathcal C_2}}}	
	\Lambda_{m_2}^{m_1}.
	\end{equation}
	
	\noindent Then, there exists a unique solution of the Dirichlet problem \cref{Dirichlet_e}, denoted by $\Theta^*(t)$.	
\end{theorem}

\begin{proof}
	Using the fact that $a/r\le 1/(1-e)$ by \cref{r}, equations \cref{cond_un} and \cref{alpha} imply that $1>||\tilde A(t,\Theta)||$ for all $(t,\Theta)\in \mathds R^3$, where we use the maximum norm. Furthermore, if $\rho (A)$ is the spectral radius of $A$, the well known inequality $ ||\tilde A(t,\Theta)||\ge \rho (\tilde A(t,\Theta))$ guarantees that $ \gamma \mathds 1\ge \tilde A(t,\Theta)$ for some $\gamma<1$. Then, $ \gamma \mathds 1\ge  A(t)$ for $A(t)$ defined in \cref{A_int}.  Now a direct application of \Cref{lema_A} finishes the proof.
\end{proof}

\begin{remark}	
	Note that, as in the spin-orbit problem, there are two special cases for which $\Theta^*$ can be computed explicitly for some combination of parameters satisfying \cref{cond_un}. If $\Lambda_j=0$ and $\Lambda_{m_2}^{m_1}=0$ for $m_1m_2\ne0$, for each $e\in(0,1)$ the solution is the synchronous resonance of the uncoupled system
	\begin{equation}
	\Theta^*(t)=2(t-f(t,e))\vecdos{1}{1}.
	\end{equation}
	
	\noindent On the other hand, if $e=0$ the solution is $\Theta^*(t)=0$.
\end{remark}

\subsection{Linear stability of the solution}

Now we are interested in the stability properties of the solution $\Theta^*(t)$, which should be seen as $2\pi$-periodic and odd from now on. In the following we will find a region of parameters guaranteeing stability of the (scaled) linearized system of \cref{spin_spin_mayus} at the periodic solution $\Theta^*$, say,
\begin{equation}\label{variacional_Theta}
\ddot y + A(t) y=0,
\end{equation}
\noindent where we take the symmetric matrix $A(t)$ now defined by 
\begin{equation*}
A(t)=\tilde A(t,\Theta^*(t))=\mathcal{C}^{-1/2}  \partial_\Theta F(t,\Theta^*(t))\mathcal{C}^{-1/2},
\end{equation*}
\noindent and $\tilde A(t,\Theta)$ was defined in \cref{A_t_Th}. 

Recall from \Cref{sub:model} that the \textit{conservative} spin-spin model has a time-dependent Hamiltonian structure given by \cref{ham}. The variational equations associated to periodic solutions, like \cref{variacional_Theta}, are linear Hamiltonian systems with periodic coefficients. We will abbreviate them by LPH systems\footnote{The linear system $\mathcal C \ddot y + \partial_\Theta F(t,\Theta^*(t)) y=0$ is an LPH system in the general sense. However, for simplicity, we particularize the general theory to \cref{variacional_Theta}.}. These systems have some special properties that we will use in the following. For the general theory see \cite{yak1975} or \cite{eke}. For example, assume that $\varphi$ is a Floquet multiplier of an LPH system. Then, its inverse $\varphi^{-1}$, its complex conjugate $\bar \varphi$ and $\bar \varphi^{-1}$ are also  multipliers and have the same multiplicity as $\varphi$. This is stated in Corollary 6 of Chapter 1.1 of \cite{eke}. Let us point out two interesting consequences. First, a necessary condition for stability of an LPH system is that all its Floquet multipliers must have modulus 1. Second, an LPH system can never be asymptotically stable. In order to do continuation of periodic solutions to the \textit{dissipative} regime we will need the concept of strong stability for LPH systems.

\begin{definition}
	Let $A_0(t)\in \mathds R^{d\times d}$ be a fixed symmetric and $T$-periodic matrix. Assume that the there exists a number $\varepsilon>0$ such that the equation $\ddot y +A_*(t)y=0$ is stable for all $A_*(t)\in \mathds R^{d\times d}$ symmetric and $T$-periodic satisfying $\int_0^T||A_*(t)-A_0(t)||<\varepsilon$. Then, $\ddot y + A_0(t) y=0$ is strongly stable.
\end{definition}
In other words, if an LPH system is strongly stable, then, any  sufficiently small perturbation of it is stable. The perturbation should keep the Hamiltonian structure. Let us illustrate this with an example of the so-called Mathieu equation. Consider the $2\pi$-periodic equation
\begin{equation*}
\ddot x + \frac{1}{4} (1+\epsilon \cos t)x=0,\quad x\in \mathds R.
\end{equation*}
\noindent For $\epsilon=0$ it is stable, but not strongly stable, because we can always find a small number $\epsilon\ne 0$ such that the corresponding equation is not stable. This is called parametric resonance, see \cite{arn}.

Strong stability can be characterized with the Floquet multipliers of the system. For example, take an LPH system whose multipliers belong to the unit circle. If the multiplicity of all the multipliers is one, then the system is strongly stable. However, the converse is not true. M. Krein developed a theory to determine if a system is strong stable with further algebraic properties of the multipliers. For our purpose of making continuation of periodic solutions the following property is relevant.

\begin{proposition}\label{propo:1}
	Assume that $\ddot y + A(t)y=0$, with $A(t)\in \mathds R^{d\times d}$ symmetric and $T$-periodic, is strongly stable. Then, neither $1$ nor $-1$ are Floquet multipliers of the system.
\end{proposition}

We will not prove this property because it is a particular result of the general theory. Nonetheless, it can be inferred by the paragraph previous to Theorem 10 in Chapter 1.2 of \cite{eke}, that is the main result of Krein's theory.

Some sufficient conditions for strong stability of \cref{variacional_Theta} are given by the following Lyapunov-like stability criterion, from Test 4, in \cite{yak1975}, Chapter III, Section 7.

\begin{stability}
	The equation $\ddot y + A(t)y=0$, with $A(t)\in \mathds R^{d\times d}$ symmetric and $2\pi$-periodic, is strongly stable provided that, for all $x\in \mathds R^d \backslash\{0\}$,
	\begin{equation}\label{cond_lin_est}
	\int_0^{2\pi} \langle A(t)x,x \rangle \d t >0 \quad \text{and} \quad \int_0^{2\pi} \Tr (A(t))\d t <\frac{2}{\pi}.
	\end{equation}
\end{stability}

This stability test is the main tool for the proof of the next theorem.

\begin{theorem}\label{th_estabilidad}
	Assume that the parameters of the model satisfy the following conditions.
	\begin{equation}\label{cond_lin_est1}
	\frac{1}{\pi^2}>
	\frac{1}{(1-e)^3} \parentesis{\frac{\Lambda_1}{\mathcal C_1}+\frac{\Lambda_2}{\mathcal C_2}} 
	+ \frac{1}{(1-e)^5} \sum_{(m_1,m_2)\in \Xi }\parentesis{\frac{m_1^2}{\mathcal C_1}+\frac{m_2^2}{\mathcal C_2}}	
	\Lambda_{m_2}^{m_1},
	\end{equation}
	\begin{equation}\label{cond_lin_est2}
	\frac{1}{4\pi} > M := \frac{1}{(1-e)^3} \max \left \{\frac{\Lambda_1}{\mathcal C_1},\frac{\Lambda_2}{\mathcal C_2} \right \}
	+\frac{1}{(1-e)^5} \sum_{(m_1,m_2)\in \Xi }\max \left \{\frac{|m_1|}{\mathcal C_1},\frac{|m_2|}{\mathcal C_2} \right \} \Lambda_{m_2}^{m_1} + \frac{4e\sqrt{1-e^2}}{(1-e)^4},
	\end{equation}
	\begin{equation}\label{cond_lin_est3}
	\cos (2\pi^2 M ) \min \left \{\frac{\Lambda_1}{\mathcal C_1},\frac{\Lambda_2}{\mathcal C_2} \right \} > \max \left \{\alpha_1\frac{\Lambda_1}{\mathcal C_1},\alpha_2 \frac{\Lambda_2}{\mathcal C_2} \right \}
	,
	\end{equation}
	
	\noindent with $\alpha_j$ defined in \cref{alpha}. Then the solution $\Theta^*(t)$ is strongly linearly stable.
\end{theorem}

Note that the second condition of \cref{cond_lin_est} is guaranteed by \cref{cond_lin_est1}. The first condition of \cref{cond_lin_est} is a bit more complicated, but its proof is immediate by the following two lemmas.

\begin{lemma}\label{lema_bound_sol}
	The components of the solution $\Theta^*(t)$ satisfy the following bounds $|\Theta^*_j(t)|\le 2 \pi^2 M$, $|\dot \Theta^*_j(t)|\le 2 \pi M$ provided that $M\ge ||\mathcal C^{-1} F(t,\Theta^*(t))||$.
\end{lemma}

\begin{proof}
	Integrating the identity $\ddot\Theta^*(t)+\mathcal C^{-1} F(t,\Theta^*(t))=0$ and taking the first component,
	\begin{equation*}
	\dot\Theta_1^*(t) = \dot\Theta_1^*(t_0) -\int_{t_0}^t u_1 \, \mathcal C^{-1} F(s,\Theta^*(s)) \d s,
	\end{equation*}
	
	\noindent where $u_1$ is the row vector $(1, 0)$. Then, for $t\in[t_0,t_0+2\pi]$,
	\begin{equation*}
	|\dot\Theta_1^*(t)|  \le |\dot\Theta_1^*(t_0)| + \int_{t_0}^{t_0+2\pi} ||\mathcal C^{-1} F(s,\Theta^*(s))||\d s \le  |\dot\Theta_1^*(t_0)| +2\pi M,
	\end{equation*}
	
	\noindent where $||\cdot ||$ indicates a matrix norm induced by a norm in $\mathds R^2$. Since $\Theta_1^*(t)$ is $2\pi$-periodic, we can choose $t_0$ such that $ \dot\Theta_1^*(t_0)=0$. The same is applicable to $\Theta_2$ for a possibly different $t_0$, consequently, $|\dot\Theta_j^*(t)|  \le 2\pi M$ for all $t$. Furthermore, since $\Theta_1^*(0)=0$,
	\begin{equation*}
	\Theta_1^*(t) = \int_{0}^t \dot\Theta_1^*(s) \d s,
	\end{equation*}
	
	\noindent and, due to the odd symmetry of $\Theta_1^*(t)$, it is enough to consider $t\in[0,\pi]$. Then, $|\Theta_1^*(t)|\le 2\pi^2M$. The same is true for $\Theta_2^*(t)$.
\end{proof}

\begin{lemma}\label{lema_def_pos}
	The conditions \cref{cond_lin_est2} and \cref{cond_lin_est3} are sufficient so that  $A(t)=\tilde A(t,\Theta^*(t))\ge \gamma \id$ for some $\gamma>0$.
\end{lemma}

\begin{proof}
	The proof this lemma is based on the following fact. Considering the partial ordering of symmetric matrices given by \Cref{def:sym}, the conditions \cref{cond_lin_est2} and \cref{cond_lin_est3} imply that the term proportional to $1/r^3$ in \cref{A_t_Th} dominates the other term, that is proportional to $1/r^5$. Let us prove it. We can compute the derivatives of $f(t)$ using \Cref{t,r,f} and get
	\begin{equation*}
	\ddot f(t)=-\frac{2e\sqrt{1-e^2} \sin (u(t))}{(1-e\cos (u(t)))^4},
	\end{equation*}
	
	\noindent where $u$ is the eccentric anomaly. Using the maximum norm we see from \cref{cond_lin_est2} and \cref{F_Theta} that $1/(4\pi)>M\ge ||\mathcal C^{-1} F(t,\Theta^*(t))||$. Furthermore, from \Cref{lema_bound_sol} we know that $|\Theta^*_j(t)|\le 2 \pi^2 M$, then, we can see graphically that
	\begin{equation*}
	\cos \Theta_1^*(t) \ge \cos (2\pi^2 M ) > 0,
	\end{equation*}
	
	\noindent therefore,
	\begin{equation}\label{order_r3}
		\matdos{\frac{\Lambda_1}{\mathcal C_1} \cos \Theta_1^*(t)}{0}{0}{\frac{\Lambda_2}{\mathcal C_2} \cos \Theta_2^*(t)} \ge \cos (2\pi^2 M )  \min \left \{\frac{\Lambda_1}{\mathcal C_1},\frac{\Lambda_2}{\mathcal C_2} \right \}  \mathds 1.
	\end{equation}
	
	On the other hand, let us define
	\begin{equation*}
	B=-\parentesis{\frac{a}{ r(t)}}^{2}
	\sum_{(m_1,m_2)\in \Xi } \matdos{\frac{m_1^2}{\mathcal C_1}}{\frac{m_1m_2}{\sqrt{\mathcal C_1\mathcal C_2}}}{\frac{m_1m_2}{\sqrt{\mathcal C_1\mathcal C_2}}}{\frac{m_2^2}{\mathcal C_2}} \,\Lambda_{m_2}^{m_1}	 
	\cos(m_1\Theta_1^*(t)+m_2 \Theta_2^*(t)).
	\end{equation*}
	
	\noindent As we did in the Proof of \Cref{propo_unicidad}, we can take the maximum norm and obtain that 
	\begin{equation*}
	\max\left \{\alpha_1\frac{\Lambda_1}{\mathcal C_1},\alpha_2 \frac{\Lambda_2}{\mathcal C_2} \right \} \ge ||B||\ge\rho(B)
	\end{equation*}
	
	\noindent where $\rho(B)$ is the spectral radius of $B$, then,
	\begin{equation*}
	 \max \left \{\alpha_1\frac{\Lambda_1}{\mathcal C_1},\alpha_2 \frac{\Lambda_2}{\mathcal C_2} \right \} \mathds 1 \ge B.
	\end{equation*}

	\noindent From this inequality, \cref{order_r3} and the definition \cref{A_t_Th} of  $\tilde A(t,\Theta)$, we prove that $\tilde A(t,\Theta^*(t))\ge \gamma \mathds 1$  with
	\begin{equation*}
	\gamma=\cos (2\pi^2 M )  \min \left \{\frac{\Lambda_1}{\mathcal C_1},\frac{\Lambda_2}{\mathcal C_2} \right \} - \max \left \{\alpha_1\frac{\Lambda_1}{\mathcal C_1},\alpha_2 \frac{\Lambda_2}{\mathcal C_2} \right \}>0.
	\end{equation*}
\end{proof}

Now we see that \Cref{lema_def_pos} implies the first condition of \cref{cond_lin_est} because $\langle A(t)x,x \rangle\ge \gamma ||x||^2>0 $.

\section{The synchronous resonance in the dissipative regime}\label{sec:asymp}%%%%%%%%%%%%%%%%%%%%%%%%%%

Recall from \cref{spin_spin_dis_mayus} that the \textit{dissipative} spin-spin model takes the form of the system
\begin{equation}\label{spin_spin_dis_mayus_bis}
 \ddot \Theta +  \diag (\delta) D(t) \dot \Theta + \mathcal C^{-1}F(t,\Theta) = 0   ,\qquad \delta=\vecdos{\delta_1}{\delta_2},\enspace \delta_j\ge 0,
\end{equation}

\noindent with $D(t)=(a/r(t))^6$. We know from \Cref{th_estabilidad} that, for $\delta=0$, there exists an odd $2\pi$-periodic solution $\Theta^*(t)$, that is strongly linearly stable in the set of the parameters space satisfying the conditions given in \Cref{cond_lin_est1,cond_lin_est2,cond_lin_est3}.

The main result of this section is \Cref{teo_asint}. There we will see that the \textit{conservative} periodic solution $\Theta^*(t)$ can be continued in the presence of friction to an asymptotically stable periodic solution $\Psi^*(t,\delta)$. However, the odd symmetry of the solution is lost because \cref{spin_spin_dis_mayus_bis} is not invariant under the change $(t,\Theta)\rightarrow (-t,-\Theta)$ as in the \textit{conservative} case. The proof of \Cref{teo_asint} is mainly based on Theorem 2 in \cite{misort2020} and on classical results on continuation of periodic solutions summarized in the next proposition.

\begin{proposition}\label{propo:cod}
	Let $\mathcal F$ be a real analytic function $\mathcal F= \mathcal F(t,x,\zeta)$, such that $\mathcal F(t+T,x,\zeta)= \mathcal F(t,x,\zeta)$, with $t\in\mathds R$, $ x\in \mathds R^n$, $\zeta\in \mathds R^d$. Assume that the equation $\dot x = \mathcal F(t,x,0) $ has a $T$-periodic solution $x=p(t)$. \begin{enumerate}
		\item Suppose that $1$ is not a Floquet multiplier of the corresponding variational equation at $x=p(t)$,
		\begin{equation*}
		\dot y = \partial_x\mathcal F(t,p(t),0) y.
		\end{equation*}
		
		\noindent Then, for $\zeta\ne0$, with small enough norm $||\zeta||$, the equation $\dot x = \mathcal F(t,x,\zeta) $ has a $T$-periodic solution $x=p_c(t, \zeta)$ such that $p_c(t, 0)=p(t)$. Moreover, $p_c(t, \zeta)$ is an analytic function and it is unique of each $\zeta$.
		\item If additionally, $p(t)$ is asymptotically stable, then this is also true for $p_c(t, \zeta)$.
	\end{enumerate}
\end{proposition}

For the detailed proof of this proposition, see Theorems 1.1 and 1.2 in Chapter 14, \cite{cod}. Now we can state the main theorem.

\begin{theorem}\label{teo_asint}
	Assume that the parameters of the system satisfy the conditions in \Cref{th_estabilidad}. If $|\delta_j|$ are small enough, then there exists a function $\Psi^*(t,\delta)$, analytic in both entries, satisfying
	\begin{enumerate}
		\item[i)] $\Psi^*(t,0)= \Theta^* (t)$ for each $t\in \mathds R$.
		\item[ii)] $\Psi^*(t,\delta)$ is a $2\pi$-periodic solution of \cref{spin_spin_dis_mayus_bis}. Moreover, if $|\Lambda_{m_2}^{m_1}|$ are small enough, then, $\Psi^*(t,\delta)$ is asymptotically stable.
	\end{enumerate}
\end{theorem}

\begin{proof}
	Recall that the \textit{conservative} periodic solution $\Theta^* (t)$ is strongly linearly stable. \Cref{propo:1} guarantees that $1$ is not a Floquet multiplier of the variational equation at $\Theta^* (t)$. Then, we can apply the first item of \Cref{propo:cod} to make the analytic continuation of the periodic solution from the \textit{conservative} ($\delta_j=0$) to the \textit{dissipative} regime ($\delta_j> 0$). We conclude that there exists a unique analytic $2\pi$-periodic solution $\Psi^*(t,\delta)$ of \cref{spin_spin_dis_mayus_bis} such that $\Psi^*(t,0)= \Theta^* (t)$ for small enough $\delta_j$.
	
	Let us explain more in detail the proof that the continuation is asymptotically stable. If $\Lambda_{m_2}^{m_1}=0$ for all $({m_1},{m_2})\in \Xi$, then \cref{spin_spin_dis_mayus_bis} takes the form of two uncoupled \textit{dissipative} spin-orbit equations
	
	\begin{equation}\label{spin-orbit_j_dis}
	\ddot \Theta_j + \delta_j \parentesis{\frac{a}{ r(t)}}^6 \dot \Theta_j + \frac{\Lambda_j}{\mathcal C_j} \parentesis{\frac{a}{ r(t)}}^3 \sin \Theta_j =0.
	\end{equation}
	
	Besides, conditions in \Cref{cond_lin_est1,cond_lin_est2,cond_lin_est3} guarantee that, for $\Lambda_{m_2}^{m_1}=0$, the \textit{conservative} solution $\Theta^*(t)$ is strongly linearly stable. We can see the solution $\Theta^*(t)$ split in two components $\Theta_j^*(t)$, each of them is a solution of the \textit{conservative} spin-orbit problem \cref{spin-orbit_j_dis} with $\delta_j=0$. Now we can apply Theorem 2 in \cite{misort2020} that guarantees that each equation in \cref{spin-orbit_j_dis} has an asymptotically stable $2\pi$-periodic solution $\Theta_{j,\delta_j}^*(t)$ provided that $\delta_j\in (0,\bar \delta_j]$. Here $\bar \delta_j$ are small numbers quantified in \cite{misort2020}. Moreover, $\Theta_{j,\delta_j}^*(t)$ is the unique continuation of $\Theta_j^*(t)=\Theta_{j,0}^*(t)$.
	
	Let us consider \cref{spin-orbit_j_dis} as a system of two equations. This system has an asymptotically stable $2\pi$-periodic solution  $\Psi^*(t,\delta)=(\Theta_{1,\delta_1}^*(t),\Theta_{2,\delta_2}^*(t))^\mathsf{T}$ such that $\Psi^*(t,0)= \Theta^* (t)$. If $|\Lambda_{m_2}^{m_1}|$ are small, we can see \cref{spin_spin_dis_mayus_bis} as a perturbation of the system \cref{spin-orbit_j_dis} and apply the second item of \Cref{propo:cod}. In this way we guarantee that $\Psi^*(t,\delta)$ has a $2\pi$-periodic continuation for $\Lambda_{m_2}^{m_1}\ne0$ that is asymptotically stable if $|\Lambda_{m_2}^{m_1}|$ are small enough.
\end{proof}

Note that for asymptotic stability we require not only that $|\delta_j|$ should be small, but also $|\Lambda_{m_2}^{m_1}|$. We would like to erase this condition on the coupling parameters $\Lambda_{m_2}^{m_1}$. However, from a theoretical point of view, this is certainly difficult to address in general since we deal with systems of differential equations. Let us explain this point. The variational equation of \cref{spin_spin_dis_mayus_bis} near $\Psi^*(t,\delta)$ is
\begin{equation}\label{lin_dis}
\ddot \eta +  \diag (\delta) D(t) \dot \eta + \mathcal C^{-1}\partial_\Theta F(t,\Psi^*(t,\delta))\eta = 0   , \quad \delta=\vecdos{\delta_1}{\delta_2}.
\end{equation}

For $e\ne0$, \cref{lin_dis} is a linear $2\pi$-periodic system of two equations of second order. In \cite{misort2020}, asymptotic stability was proved for the spin-orbit problem taking advantage of the following fact. Any second order periodic equation $\ddot x + a_1(t)\dot x + a_0(t) x =0$, $x\in \mathds R$, $a_n(t)=a_n(t+T)$, can be converted into a Hill's equation $\ddot \chi + \alpha(t)\chi =0$, $\alpha(t)=a_0(t)-\frac{1}{4}a_1(t)^2- \frac{1}{2}\dot a_1(t)$, by the change of variables $\chi(t)=x(t)\exp (\frac{1}{2}\int_0^ta_1(s) \d s)$. See \cite{mag}. We can see the \textit{dissipative} problem ($a_1(t)\ne0$) as a perturbation of the \textit{conservative} one ($a_1(t)=0$). Assume that $\ddot x + a_0(t) x =0$ is strongly stable, then $\ddot \chi + \alpha(t)\chi =0$ is stable. Since it is a Hill's equation (also a LPH system), the modulus of the Floquet multipliers of $\ddot \chi + \alpha(t)\chi =0$ is 1. Now we undo the change of variables and conclude that the modulus of the Floquet multipliers of $\ddot x + a_1(t)\dot x + a_0(t) x =0$ is smaller than 1, therefore, it is asymptotically stable. However, it is not clear how to perform an analogous procedure in \cref{lin_dis}. The main obstacle is the non-commutativity of matrices due to the asymmetric nature of the \textit{dissipative} problem ($\delta_1\ne\delta_2$). Actually, if we follow the same steps, we end up with a system of equations that is no longer periodic for $\delta_1\ne\delta_2$. The numbers $\delta_j$ depend on several parameters of the bodies and we do not see any good physical reason to impose both dissipative parameters to be equal. In fact, if $\delta_1\ne\delta_2$, in principle the \textit{dissipative} spin-spin model cannot be considered conformally symplectic as the spin-orbit problem. See \cite{calceldel}. From this discussion, we conclude that this it is necessary a deeper theoretical study, but it is beyond the scope of this paper.

On the other hand, let us see that for $e=0$, the solution of \cref{lin_dis} is asymptotically stable. The solution given by \Cref{teo_asint} is $\Psi^*(t,\delta)\equiv 0$. Taking $y=\mathcal C^{1/2} \eta$, the corresponding variational equation is
\begin{equation}\label{lin_dis_cons}
\ddot y + \diag (\delta) \dot y + A y =0,
\end{equation}

\noindent where $A$ is the symmetric constant matrix given by
\begin{equation*}
A=
\matdos{\xi_1}{\sigma}{\sigma}{\xi_2}
=\matdos{\frac{\Lambda_1}{\mathcal C_1} }{0}{0}{\frac{\Lambda_2}{\mathcal C_2} }
+
\sum_{(m_1,m_2)\in \Xi } \matdos{\frac{m_1^2}{\mathcal C_1}}{\frac{m_1m_2}{\sqrt{\mathcal C_1\mathcal C_2}}}{\frac{m_1m_2}{\sqrt{\mathcal C_1\mathcal C_2}}}{\frac{m_2^2}{\mathcal C_2}} \,\Lambda_{m_2}^{m_1}	 
.
\end{equation*}

\noindent Note that, by conditions \cref{cond_lin_est2} and \cref{cond_lin_est3}, $A$ is a positive definite matrix. See \Cref{lema_def_pos}. The characteristic polynomial of equation \cref{lin_dis_cons} is
\begin{equation*}
p(\omega)=\omega^4 + (\delta_1+\delta_2 )\omega^3+(\xi_1+\xi_2+\delta_1\delta_2 )\omega^2+ (\xi_1\delta_2+\xi_2\delta_2)
\omega + \det A.
\end{equation*}

Equation \cref{lin_dis_cons} is asymptotically stable if and only if all the roots of $p(\omega)$ have negative real parts. This can be checked with the Routh-Hurwitz criterion, see \cite{gan1959}. According to it, all the roots of the polynomial have negative real parts if and only if the associated Hurwitz determinants of the polynomial are strictly positive, say,
\begin{equation*}
D_1=\delta_1+\delta_2 ,\quad D_2=\delta_1^2\delta_2+ \delta_2^2\delta_1+ \xi_1\delta_1 + \xi_2 \delta_2,
\end{equation*}
\begin{equation*}
D_3=D_1^2\sigma^2+\delta_1\delta_2(D_1(\xi_1\delta_2+\xi_2\delta_2)+(\xi_1-\xi_2)^2),\quad D_4=D_3 \det A.
\end{equation*}

\noindent Since $A$ is positive definite, we get asymptotic stability for all $\delta_1$ and $\delta_2$ such that both are non-negative and at least one is different from zero.

\section{Applications}\label{sec:app}

Recall from the end of \Cref{sec:model} that our model depends on six independent physical parameters $(e;\mathcal C_1, \lambda_1, \lambda_2, \hat d_1,\hat q_1)$, where $e$ is the orbital eccentricity, $\mathcal C_j$ the moment of inertia of $\mathcal E_j$ with respect to the $\mathsf{c}_j$-axis, $\lambda_j=\Lambda_j/\mathcal C_j$ is the oblateness of $\mathcal E_j$ in the plane of motion,  and $\hat d_j$ and $\hat q_j$ are, respectively, the oblateness and the flatness of $\mathcal E_j$ with respect to the size of the orbit.

We have two type of estimates. The first type in \cref{cond_un} guarantees uniqueness of the synchronous resonance in the \textit{conservative} regime. The second one in \Cref{cond_lin_est1,cond_lin_est2,cond_lin_est3} guarantees linear stability of the same solution. Our estimates depend on certain values $\alpha_j$ in \cref{alpha}. To write them in terms of the physical parameters, we use the definitions in \Cref{Lambda_j,Lambda_1_0,Lambda_2_0,Lambda_1_1}, then
\begin{align}
\label{sum1}
\sum_{(m_1,m_2)\in \Xi } \frac{m_j^2}{\mathcal C_j}	
\Lambda_{m_2}^{m_1} = & 
\lambda_j
\parentesis{
	\frac{25}{4}\hat d_j + \frac{25}{28}\hat q_j + \frac{19}{4}\hat d_{3-j} + \frac{5}{4}\hat q_{3-j}
},\\
\label{sum2}
\sum_{(m_1,m_2)\in \Xi } \frac{|m_1m_2|}{\sqrt{\mathcal C_1\mathcal C_2}}	
\Lambda_{m_2}^{m_1} =& \frac{19}{4} \sqrt{\frac{\mathcal C_1}{\mathcal C_2}}
\lambda_1 \hat d_2 = \frac{19}{4} \sqrt{\frac{\mathcal C_2}{\mathcal C_1}}
\lambda_2 \hat d_1 ,
\\
\label{sum3}
\sum_{(m_1,m_2)\in \Xi } \frac{|m_j|}{\mathcal C_j}	
\Lambda_{m_2}^{m_1} = &
\lambda_j
\parentesis{
	\frac{25}{8}\hat d_j + \frac{25}{28}\hat q_j + \frac{19}{4}\hat d_{3-j} + \frac{5}{4}\hat q_{3-j}
}.
\end{align}

\noindent Now we are ready to apply our estimates to specific cases.

\subsection{Real systems}

In one hand, the Pluto-Charon binary is the largest known system that is in double synchronous resonance. The physical parameters of the system relevant for the spin-spin model are shown in \Cref{tab:real}. Pluto is almost twice the size of Charon, contains the 89\% of the mass and the 97\% of the body moment of inertia ($\mathcal C_j$) of the system. Besides, the size of the orbit is quite large ($a=27.2$) compared to the sizes of the bodies. This results in very small values of $\hat d_j$ of order $10^{-7}$, which means this is a certainly weak spin-spin coupling. The orbit has a very small eccentricity $e=0.0002$. Recall that the double synchronous resonance of the circular case ($e=0$) is the trivial solution $\Theta(t)\equiv0$, both for the \textit{conservative} case \cref{spin_spin_mayus} and the \textit{dissipative} case \cref{spin_spin_dis_mayus}. The asymptotic stability of the solution for any value of the dissipative parameters is easily guaranteed, as it was shown at the end of \Cref{sec:asymp} using equation \cref{lin_dis_cons}. For the real eccentricity, the solution $\Theta^*(t)$ of \cref{spin_spin_mayus} oscillates very close to zero and our estimates guarantee the uniqueness and linear stability of  solution. Furthermore, \Cref{teo_asint} shows the existence of an asymptotically stable solution $\Psi^*(t,\delta)$ of the \textit{dissipative} model provided that $\delta_j$, $\hat d_j$ and $\hat q_j$ are small enough. Unfortunately, this last result is not quantified in this paper for the real parameters.

\begin{table}
	\begin{center}
		\begin{tabular}{|c||cccccc|cc|}
			\hline
			System   &  $M_j$ & $\mathsf{a}_j$ &$\mathcal C_j$    &  $ \lambda_j$ & $\hat d_j$         & $\hat q_j$ & $a$ & $e$ \\ \hline  \hline \hline %&&\\[-1em]
			Pluto &  0.89    & 1.65   & $0.97$ & $ 3.3\cdot 10^{-5} $  & $1.5\cdot 10^{-7}$ & 	$1.2\cdot 10^{-6}$ & \multirow{2}{*}{$27.2$} & \multirow{2}{*}{$2.0 \cdot 10^{-4}$} \\[0.5ex]  
			Charon  & 0.11 & 0.84    & $0.03$ &  $2.4\cdot 10^{-3}$     & $3.5\cdot 10^{-7}$ & $8.2\cdot 10^{-7}$  & & \\	\hline   \hline    
			Patroclus &  0.56    & 1.7   & $0.60$ & $ 0.11 $  & $2.6\cdot 10^{-4}$ & 	$1.2\cdot 10^{-3}$ & \multirow{2}{*}{$18.2\pm 0.5$} & \multirow{2}{*}{$0.02\pm0.02$} \\[0.5ex]  
			Menoetius  & 0.44 & 1.6    & $0.40$ &  $0.14$     & $2.2\cdot 10^{-4}$ & $9.9\cdot 10^{-4}$  & &  \\	\hline                         
		\end{tabular}
		\caption{Real physical parameters for two binary systems. For Pluto and Charon, we take the largest values of $\lambda_j$, $\hat d_j$ and $\hat q_j$ obtained from data in \cite{kho2020}. The parameters of Patroclus and Menoetius are obtained from data in \cite{davsch2020} and the orbital parameters from \cite{mar2006}.}
		\label{tab:real}
	\end{center}
\end{table}

On the other hand, the Trojan binary asteroid 617 Patroclus is a system whose components are of similar size, mass and moment of inertia. See the physical parameters of its components, Patroclus and Menoetius, in \Cref{tab:real}. Each body has a diameter of around one hundred kilometres, almost ten times smaller than Charon. Patroclus and Menoetius have a more oblate ellipsoidal shape than Pluto and Charon and the size of the orbit in this case ($a=18.2\pm0.5$) is smaller. In consequence, the corresponding dynamical parameters $\lambda_j$, $\hat d_j$ and $\hat q_j$ are several orders of magnitude larger. The orbital eccentricity is not measured with enough precision, $e=0.02\pm 0.02$. With our estimates, we are able to guarantee the uniqueness of the solution $\Theta^*(t)$ of \cref{spin_spin_mayus} for eccentricities up to $e=0.04$. However, we fail to guarantee linear stability even for $e=0$. The main reason is that the stability test given by the conditions \cref{cond_lin_est} is not fine enough for such large values of $\lambda_j$. In the following subsection we will explain what is the range of parameters that is covered by our study.

\subsection{Stability diagrams in the space of parameters}

Note that all the terms appearing in \Cref{sum1,sum2,sum3} are positive. Since $\hat q\ge \hat d$, and, in order to reduce the parameters in the upper bounds for the expressions in \Cref{sum1,sum2,sum3}, we can take $\hat d_j=\hat q_j$. In this way, we reduce the independent parameters to five $(e;\lambda_1,\lambda_2,\mathcal C_1, \hat q_1)$. Note now that, to take $\hat q_1=0$ is equivalent to break the coupling of the system, resulting in two independent spin-orbit problems.

We will consider two special cases with three free parameters. In one hand, the case of identical bodies, that we compare with the asteroid 617 Patroclus. Here the parameters are $e$, $\lambda_j=\lambda$ and $\hat q_j=\hat q$. On the other hand, the case when $\mathcal E_1$ is twice the size of $\mathcal E_2$, that we compare with the Pluto-Charon system. Here we consider the same density and the free parameters are $e$, $\lambda_2$ and $ \hat q_1$, whereas the dependent parameters are $\lambda_1=2^{-3}\lambda_2$ and $\hat q_2=2^{-5}\hat q_1$.

\Cref{fig:both} shows regions in the space of parameters for which there is uniqueness and linear stability of the double synchronous resonance according to our theoretical estimates. We see that we cover the Patroclus-Menoetius system (top panels) only for the uniqueness of the solution but not for the linear stability. In contrast, the Pluto-Charon system (bottom panels) is covered for linear stability as well. We can compare the diagrams of $\hat q=0$ and $\hat q_1=0$ with the theoretical estimates obtained in \cite{misort2020}, shown in \Cref{fig:stab_spin_orbit}. We see that, although the uniqueness region is similar, the stability region (in yellow) is considerably larger in \Cref{fig:stab_spin_orbit} than those in \Cref{fig:both}. This shows that the mathematical techniques used in \cite{misort2020} are much finer than in this paper. In \cite{misort2020} we used generalized Lyapunov criteria using $L^p$-norms, with $p\in[1,\infty]$, see \cite{zha}, and upper and lower solutions to bound the amplitude of the solution. Instead, in this paper we use the stability test given by \cref{cond_lin_est}, that is of type $L^\infty$, and a rougher bound for the amplitude of the solution in \Cref{lema_bound_sol}. Since the model is quite new, here we initiate the analysis with a simpler approach. Besides, the mathematical tools are not as well developed for systems of equations as for standard second order scalar equations.

\begin{figure}
	\centering
	\scalebox{0.9}{\includegraphics{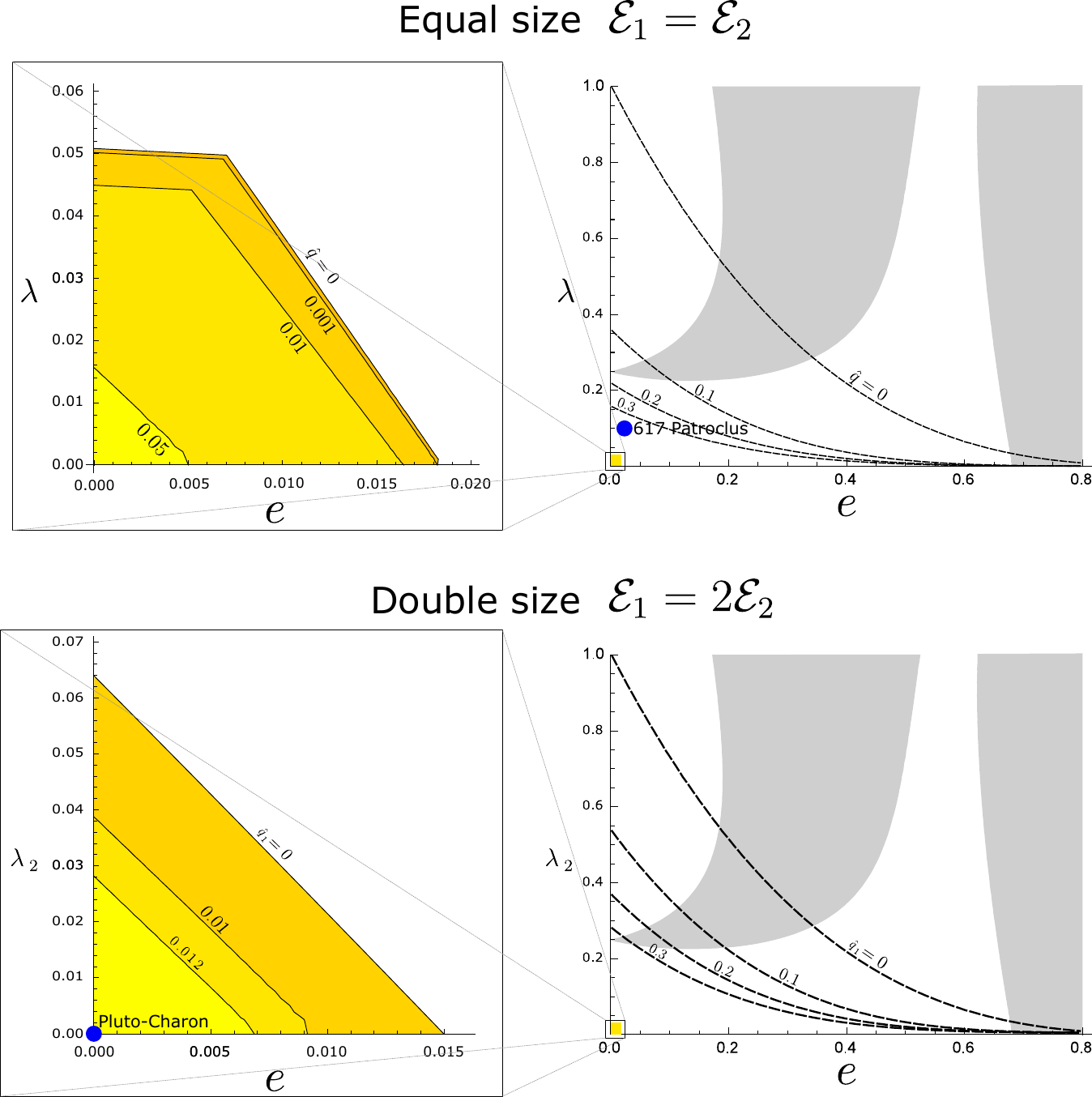}}
	\caption{Stability diagrams in the $(e,\lambda)$-plane of the synchronous resonance of the spin-spin model. Top: both bodies are equal. Bottom: one body is double the size of the other. The double synchronous resonance is unique under the dashed lines (right) and linearly stable under the black lines (left) for the indicated value of $\hat q$. In the left we see zoomed views of the stable regions. The more yellow is the region indicates that stability is guaranteed for larger values of $\hat q$. The gray regions in the right are unstable for the uncoupled system (spin-orbit), i.e., with $\hat q=0$.}
	\label{fig:both}
\end{figure}

\begin{figure}
	\centering
	\scalebox{0.9}{\includegraphics{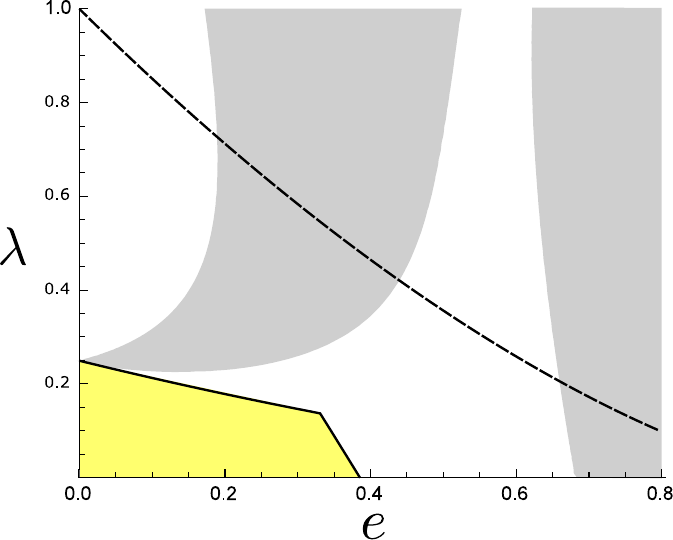}
	}
	\caption{Stability diagram in the $(e,\lambda)$-plane of the spin-orbit in \cite{misort2020}, Figure 3.}
	\label{fig:stab_spin_orbit}
\end{figure}

\begin{figure}
	\centering
	\scalebox{0.5}{\includegraphics{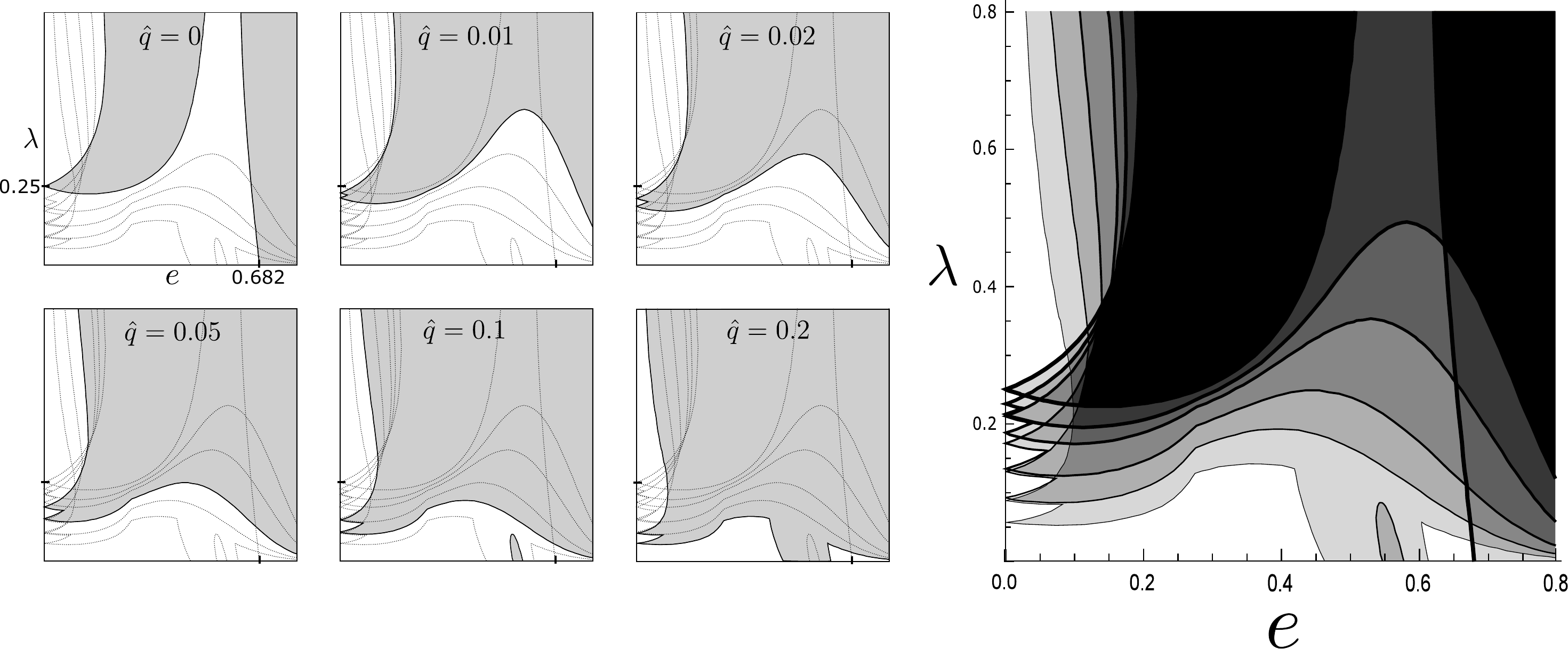}
	}
	\caption{Stability diagrams in the $(e,\lambda)$-plane in the case of equal bodies. The six plots in the left show the unstable region in gray for different values of $\hat q$. The image in the right shows the six diagrams superimposed. Darker tones of gray indicate more overlapping between unstable regions.}
	\label{fig:stab_several}
\end{figure}

We see in \Cref{fig:both} that an increase in the value of $\hat q$ results in a global reduction of the regions that we estimated theoretically, both for stability and uniqueness regions. This behavior can be compared with the numerical plots in \Cref{fig:stab_several}. We focus only on the case of equal bodies. Here we see how the instability region changes when we increase $\hat q$. There are some interesting phenomena.

\begin{enumerate}
	\item For $\hat q=0$ there is only one bifurcation point for the unstable solution in the $\lambda$-axis at $(e,\lambda)=(0,0.25)$. However, for $\hat q>0$, it becomes two bifurcation points at $(0,\lambda_{(1)})$ and $(0,\lambda_{(2)})$, with $0<\lambda_{(1)}<\lambda_{(2)}<0.25$. This opens a small window of stability at the points $(e,\lambda)$ with $e$ close to $0$ and $\lambda\in (\lambda_{(1)},\lambda_{(2)})$.
	\item For $\hat q=0$, apart from the instability region bifurcating from the $\lambda$-axis, there is another one bifurcating from the $e$-axis at $(e,\lambda)\approx(0.682,0)$. The existence of such bifurcation was studied in \cite{misort2020}. However, for $\hat q>0$, it looks that the last bifurcation point moves to the right, at the same time that the two instability regions merge into a single one. This shows that turning on the coupling has a stabilizing effect of the synchronous resonance for large $e$ and small $\lambda$. This holds up to a critical $\hat q\in (0.05,0.1)$ for which another unstable region bifurcates from the $e$-axis. This region merges with the large one at some $\hat q\in (0.1,0.2)$. This leaves an island of stability for large $e$ and small $\lambda$.
	\item In the right panel of \Cref{fig:stab_several} we see that there are some regions (the darkest ones), that remain unstable, not very affected by changes in $\hat q$. Instead, the lighter regions show more susceptibility to change their stability when $\hat q$ changes.
\end{enumerate}

In \Cref{fig:stab_several} we have taken large values of $\hat q$, compared to the real values in \Cref{tab:real}. From its definition in \cref{hat_d_q} and \cref{q_d} we see that $\hat q\le 1/a^2$ for equal bodies ($M=0.5,\mathcal C=0.5$). In order to be consistent with the Keplerian orbit approximation, $a$ should be quite larger than 1, that gives the scale of the objects. For example, $a$ of order $10$ would give an upper estimate of $\hat q$ of order $10^{-2}$. In consequence, for more realistic parameters, we should not consider the appearance of the additional instability region bifurcating from the $e$-axis from large $\hat q$.

\section{Discussion}\label{sec:discussion}

In this paper we have proposed a simplified mathematical model for the rotational dynamics in the Full Two-Body Problem. This model is a straightforward continuation of the spin-orbit problem. In consequence, we hope it will be of interest for physical applications as well as for theoretical studies. We have approached the problem from a theoretical point of view, but always keeping what we think is the essence of the physical problem: the dissipative effects are fundamental to explain the universe we observe today. In this sense, the spin-spin model not only broaden the scope of the spin-orbit problem in a higher dimensional phase space, but also contributes to fill the gap between the conservative and the dissipative effects considered in the spin-orbit problem. More precisely, if the \textit{dissipative} torque (of order $1/r^6$) is important in the evolution of a satellite, then, we should consider also the spin-spin interaction (of order $1/r^5$). Of course this two effects are more important when the bodies are closer to each other. In fact, in the spin-spin model the strength of the terms of order $1/r^5$ is given by parameters that compare the shape of the bodies with the size of the orbit, say, $\hat d_j$ and $\hat q_j$. In contrast, the spin-orbit problem only regards the equatorial oblateness of the satellite $d_j/\mathcal C_j$. It is reasonable to think that the different types of interactions, say, point-point, spin-orbit and spin-spin, must have their own specific relevance in different ranges of parameters. This shows that the non-Keplerian behavior of the full Lagrangian model \cref{spin-spin_V}, \cref{orbital_eqs}, should be investigated more deeply. Here the full expansion of the potential energy, given in \cref{V_full}, may also play a role. Moreover, as \cite{davsch2020} shows, non-planar oscillations around solutions of the planar problem can be studied and are of practical interest.

In the present research, we have made a brief theoretical study that allowed us to point out the importance of the double synchronous resonance and compare it with the synchronous resonance of the spin-orbit problem. Particularly, in a similar way than \cite{misort2020}, we determine sufficient conditions for the existence of an asymptotically stable periodic solution (capture into resonance). Besides, note that our estimates do not pretend to be optimal at all. Instead, we illustrate a way to extend to the spin-spin model the tools used for the spin-orbit model, as well as to compare them. Furthermore, in this sense we have included some numerical diagrams of linear stability in \Cref{fig:stab_several} that show us how the spin-spin interaction alters the schemes of the spin-orbit model.

We have applied our study to two real systems in double synchronous resonance. In one hand, Pluto and Charon are representative of a large binary with one body much larger than the other one, see \cite{dob1997}. On the other hand, the binary asteroid 617 Patroclus is an archetype of a small system of similar components, see \cite{davsch2020}, \cite{mar2006}. Here we propose a way how to make an effective comparison between different systems. Note that the convenient choice of units and parameters helps to clarify the comparison. As we expected, the best candidates to apply the spin-spin model are binary asteroids. They are very abundant in the solar system, e.g., about 15\% of the near-Earth asteroids are thought to be binaries. For a detailed discussion on the applications of the general spin-spin model and its full Lagrangian version, we refer to \cite{batmor2015} and the bibliography therein. With our study on the double synchronous resonance we hope to contribute to the study of the spin-spin resonances made in \cite{batmor2015}. Whereas they focus on the synchronization of both spins for slow circular orbital motion ($\dot f \ll \dot \theta_j$), we consider the full synchronization including the orbit with arbitrary eccentricity. According to \cite{davsch2020}, most of the equal mass binaries are expected to be in the double synchronous state. In \cite{mur}, Section 4.14, they provide a formula for a critical mass ratio of the components for this state to be possible. We want to remark also that, apart from the application to binary asteroids and large natural satellites, the spin-spin interaction can be relevant for artificial satellites whose rotation state along an orbit is important. For instance, communication satellites in equatorial orbits or even spacecraft exploring small bodies.

Finally, we think that the theoretical interest of the model is large, even beyond the phenomena already observed in the spin-orbit problem. For example, in the spin-orbit problem we can apply the notion of KAM stability because KAM tori confine regions in the phase space. However this does not happen in the spin-spin model due to the increase in the phase space dimension (two degrees of freedom and time dependence). In fact, it is expected that Arnold diffusion takes place in this case. In general, the weak coupling and the Hamiltonian character of the system makes it suitable to apply perturbative techniques. Particular questions may be investigated, such as chaos by overlapping of resonances, stochastic phenomena, normally hyperbolic manifolds, scattering maps, among other phenomena, see \cite{chi1979}.

\appendix %%%%%%%%%%%%%%%%%%%%%%%%%%%%%%%%%%%%%%%%%%%%%%%%%%%%%%%%%%%%%%%%%%%%%%%%%%%%%%%%%%%%%%

\section{Units}\label{ap:units}

If $t$, $M$ and $l$ stand for time, mass and length respectively, the relation between our system of units and any other one is the following
\begin{equation*}
t_{\text{ours}} = \frac{2\pi}{T} t, \quad M_{\text{ours}} = \frac{M}{M_1+M_2} , \quad l_{\text{ours}} = l\sqrt{ \frac{M_1+M_2}{\mathcal C_1+\mathcal C_2}} \, .
\end{equation*}

\noindent It is worth mentioning that, if $I$ is any magnitude with units of moment of inertia, then the conversion is given simply by
\begin{equation*}
 I_{\text{ours}} =\frac{I}{\mathcal C_1+\mathcal C_2} .
\end{equation*}

The value of the gravitational constant $G$ in any system of units must respect Kepler's third law \cref{kepler3}. 

\section{Derivation of the potential of the spin-spin problem}\label{app:potential}

\subsection{Potential of the Full Two-Body Problem} \label{sub:V_full}

The expansion of the potential energy in the Full Two-Body Problem has been obtained in several papers, see \cite{tri2008} for example. In this subsection, and in order to introduce some notation, we present a short derivation of the spherical harmonics expansion, following the approach of \cite{mac1995} and \cite{bou2017}. See also a similar approach in \cite{matlep2009} and \cite{comlem2014}. We start from the formula
\begin{equation*}
V=-G\int \int \frac{\d M_1(\vec x_1) \, \d M_2(\vec x_2)}{|\vec x_1-\vec x_2|},
\end{equation*}

\begin{figure}
	\centering
	\scalebox{0.8}{\includegraphics{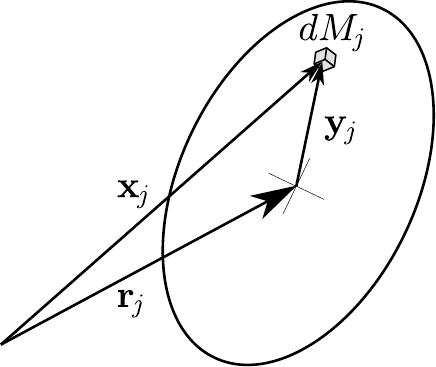}}
	\caption{Position vectors of the element of mass of an ellipsoid.}
	\label{fig:spin-spin_pot}
\end{figure}

\noindent where each $\vec x_j\in \mathds R^3$ is the position vector (with respect to the barycenter of the system) of the mass element $\d M_j(\vec x_j)$ corresponding to the ellipsoid $\mathcal E_j$. Making the change of variables $\vec y_j=\vec x_j-\vec r_j$, illustrated in \Cref{fig:spin-spin_pot}, and defining $\vec y=\vec y_1-\vec y_2$, we obtain
\begin{equation*}
V=-G \int \int \frac{d M_1(\vec y_1) \, dM_2(\vec y_2)}{|\vec r - \vec y|}.
\end{equation*}

\noindent Recall that $\vec r =\vec r_2 -\vec r_1 $. The usual expansion in spherical harmonics gives us
\begin{equation}\label{V_expan}
V=-G\sum_{(l,m)\in \Upsilon}  Q_{l,m}  \frac{Y_{l,m}(\hat {\vec r}) }{|\vec r |^{l+1}},
\end{equation}

\noindent where
\begin{equation*}
\Upsilon= \{ (l,m)\in \mathds Z^2 : \ 0\le |m| \le l    \}
\end{equation*}

\noindent and the multipolar moments of the system $Q_{l,m}$ are defined by
\begin{equation}\label{Q_lm_y}
Q_{l,m}=\int \int |\vec y|^l \bar Y_{l,m}(\hat {\vec y}) \d M_1(\vec y_1) \, \d M_2(\vec y_2),
\end{equation}

\noindent where the upper bar indicates complex conjugation. We use the Schmidt semi-normalization\footnote{With this choice, the Legendre polynomials can be written in terms of the spherical harmonics as
	\begin{equation*}
	P_l( \hat {\vec r} \cdot \hat {\vec y} )= \sum_{m=-l}^l Y_{l,m}(\hat {\vec r}) \bar Y_{l,m}(\hat {\vec y}). 
	\end{equation*}
} of the spherical harmonics in the same way as in \cite{bou2017}. Assume that, in the inertial frame, $\vec{r}$ has spherical coordinates $(r,\vartheta,\phi)$, then, the spherical harmonics are defined by
\begin{equation*}
Y_{l,m}(\vartheta,\phi) = (-1)^m \sqrt{\frac{(l-m)!}{(l+m)!}} P_{l,m} (\cos \vartheta) \exp(i m\phi),
\end{equation*}

\noindent where the associated Legendre polynomials are given by
\begin{equation*}
P_{l,m}(x)=\frac{1}{2^l l!} (1-x^2)^{m/2} \frac{\d^{l+m}}{\d x^{l+m}} (x^2-1)^l,\quad x\in[-1,1].
\end{equation*}

Note that, since $ {\vec y}=\vec y_1-\vec y_2$, we cannot factorize the integral in \cref{Q_lm_y} into factors that involve quantities associated to each body separately. However we can express this integral as a sum of factorized terms. For this we can define the auxiliary normalized solid harmonics 
\begin{equation*}
\mathcal Y_{l,m}(\vec x) = \frac{|\vec x|^l Y_{l,m}(\hat{\vec x})}{\sqrt{(l-m)!(l+m)!}}, \quad \vec x \in \mathds R^3,
\end{equation*}

\noindent and apply the translation formula, given in equation (313) in \cite{ste1973},
\begin{equation*}
\mathcal Y_{l,m}(\vec y_1-\vec y_2)  = \sum_{\lambda_1,\mu_1} \sum_{\lambda_2,\mu_2} \mathcal Y_{\lambda_1,\mu_1}(\vec y_1)\mathcal Y_{\lambda_2,\mu_2}(-\vec y_2) ,
\end{equation*}

\noindent where $\lambda_j$ and $\mu_j$ are integers running all the values such that
\begin{equation*}
0\le \lambda_j \le l, \quad \lambda_1+\lambda_2=l; \quad -\lambda_j\le \mu_j \le \lambda_j, \quad \mu_1+\mu_2=m.
\end{equation*}

\noindent Then, using the parity relation $ Y_{l,m}(-\hat{\vec x})=(-1)^l Y_{l,m}(\hat{\vec x})$, the expression \cref{Q_lm_y} becomes
\begin{equation}\label{Q_lm_decomp}
\frac{Q_{l,m}}{\sqrt{(l-m)!(l+m)!}}=  \sum_{\lambda_1,\mu_1} \sum_{\lambda_2,\mu_2} (-1)^{\lambda_2} \frac{M_1 R_1^{\lambda_1} Z^{1)}_{\lambda_1,\mu_1}}{\sqrt{(\lambda_1-\mu_1)!(\lambda_1+\mu_1)!}} \frac{M_2 R_2^{\lambda_2} Z^{2)}_{\lambda_2,\mu_2}}{\sqrt{(\lambda_2-\mu_2)!(\lambda_2+\mu_2)!}},
\end{equation}

\noindent where, the complex Stokes coefficients\footnote{The quantities $Z^{j)}_{l,m}$ provide the expansion of the potential created for the body $\mathcal E_j$. They are related to the usual parameters $C^{j)}_{l,m}$ and $S^{j)}_{l,m}$ by $$ C^{j)}_{l,m}+ i S^{j)}_{l,m} = (-1)^m \frac{2}{1+\delta_{m,0}}  \sqrt{\frac{(l-m)!}{(l+m)!}} \bar Z^{j)}_{l,m},  \quad m\ge 0, $$ where $\delta_{m,n}$ is the Kronecker delta.} of each ellipsoid are given by
\begin{equation}\label{Z_lm}
Z^{j)}_{\lambda,\mu} = \frac{1}{M_j R_j^{\lambda}}\int  |\vec y_j|^{\lambda} \bar Y_{\lambda,\mu}(\hat {\vec y}_j) \d M_j(\vec y_j) ,
\end{equation}

\noindent and $R_j$ is the mean radius of $\mathcal E_j$.

Finally, since in the potential energy the summation range is $0\le l\le \infty$, $-l\le m\le l$, which are all the possible terms, then, from \cref{V_expan} and \cref{Q_lm_decomp} we can write
\begin{equation}\label{V_expan_total}
V=-\frac{G M_1 M_2 }{|\vec r|} 
\sum_{\substack{(\lambda_1,\mu_1)\in \Upsilon \\ (\lambda_2,\mu_2)\in \Upsilon}}    
(-1)^{\lambda_2} \gamma_{\lambda_2,\mu_2}^{\lambda_1,\mu_1} 
\parentesis{\frac{R_1}{|\vec r|}}^{\lambda_1} \parentesis{\frac{R_2}{|\vec r|}}^{\lambda_2}  Z^{1)}_{\lambda_1,\mu_1} Z^{2)}_{\lambda_2,\mu_2}Y_{\lambda_1+\lambda_2,\mu_1+\mu_2}(\hat{\vec r}) ,
\end{equation}

\noindent where we defined the constants
\begin{equation*}
\gamma_{\lambda_2,\mu_2}^{\lambda_1,\mu_1} = \sqrt{ \frac{
		(\lambda_1+\lambda_2-\mu_1-\mu_2)!(\lambda_1+\lambda_2+\mu_1+\mu_2)!
	}{
		(\lambda_1-\mu_1)!(\lambda_1+\mu_1)!(\lambda_2-\mu_2)!(\lambda_2+\mu_2)!
}}.
\end{equation*}

\subsection{Potential of the ellipsoidal spin-spin model}\label{sub:V_planar}

Note that the terms in the expansion \cref{V_expan_total}, and in particular $Z^{j)}_{\lambda,\mu}$, have to be computed with respect to the inertial frame.  Let us call $\mathcal E_j$-frame to the fixed body frame of each ellipsoid, formed by its center and its principal directions associated respectively to $\mathsf a_j$, $\mathsf b_j$ and $\mathsf c_j$. Let $\mathcal Z^{j)}_{\lambda,\mu}$ be the Stokes coefficients computed with respect to the $\mathcal E_j$-frame. The $\mathcal E_j$-frame is rotated, with respect to the inertial frame, with the rotation labelled by the Euler $z$-$y$-$z$ angles $(\alpha,\beta,\gamma)=(\theta_j,0,0)$.

Let $\vec x \in \mathds R^3$ be a vector with spherical coordinates $(|\vec x|,\vartheta_j, \phi_j )$ with respect to the $\mathcal E_j$-frame and $(|\vec x|,\vartheta, \phi)$ with respect to the reference frame formed by the center of the body $\mathcal E_j$ and the directions parallel to those of the inertial frame. The relation between spherical harmonics $Y_{l,m}(\hat{\vec x})$ computed with respect to both systems of reference is the following
\begin{equation*}
Y_{l,m}(\vartheta_j,\phi_j)=\sum_{m'=-l}^l Y_{l,m'}(\vartheta,\phi) \bar D^l_{m,m'} (\alpha,\beta,\gamma)
\end{equation*}

\noindent where $ D^{l}_{m,m'} (\alpha,\beta,\gamma)$ is the $(m,m')$-element of the Wigner $D$-matrix associated to the rotation given by the Euler $z$-$y$-$z$ angles $(\alpha,\beta,\gamma)$, see \cite{ste1973}. Then, from \eqref{Z_lm},
\begin{equation*}
Z^{j)}_{\lambda,\mu} = \sum_{\mu'=-\lambda}^{\lambda} D^{\lambda}_{\mu,\mu'} (\alpha,\beta,\gamma)\  \mathcal Z^{j)}_{\lambda,\mu'}.
\end{equation*}

From the definition of the Wigner $D$-matrices, see for instance equation (186) in \cite{ste1973}, in our planar case they are diagonal $
D^{\lambda}_{\mu,\mu'}  (\theta_j,0,0) = \delta_{\mu,\mu'} \exp({-i \mu' \theta_j})$,  where $\delta_{\mu,\mu'}$ is the Kronecker delta. Then,
\begin{equation*}
Z^{j)}_{\lambda,\mu} = \mathcal Z^{j)}_{\lambda,\mu}\exp({-i \mu \theta_j}).
\end{equation*}

Now we can express \cref{V_expan_total} in terms of $\mathcal Z^{j)}_{\lambda,\mu}$. In \cite{bal1994} an expansion of the potential created by a homogeneous ellipsoid was computed. Incidentally, a complicated general expression for $\mathcal Z^{j)}_{\lambda,\mu}$ was computed there as well. In the next Proposition we summarize some remarkable properties of those quantities.

\begin{proposition}
	Let $\mathcal Z_{\lambda,\mu}$ be Stokes coefficients of an homogeneous ellipsoid computed in its own fixed body frame. They have the following properties
	\begin{enumerate}
		\item $\mathcal  Z_{\lambda,\mu}\in \mathds R$.
		\item $\mathcal  Z_{\lambda,\mu}\equiv 0$ if either $\lambda$ or $\mu$ are odd numbers.
		\item $\mathcal  Z_{\lambda,-2n}=\mathcal Z_{\lambda,2n}$, with $n$ integer.
	\end{enumerate} 
\end{proposition}

We will not reproduce the whole proof here but it can be found in \cite{misPhD}. We just want to remark that it is based on the symmetry properties of the spherical harmonics and the geometrical symmetries of the ellipsoids.

\begin{remark}
	Regarding these properties, a convenient expression to compute numerically $\mathcal Z_{2k,2n}$, with $k\ge 0$ and $n$ integers, is
	\begin{multline*}
	\mathcal Z_{2k,2n} = \frac{3}{4 \pi R^{2k}} \sqrt{\frac{(2k-2n)!}{(2k+2n)!}} \int_{\mathsf B} \operatorname{Re}( (\mathsf aZ-i\mathsf bY)^{2n}) \frac{[(\mathsf aX)^2+(\mathsf bY)^2+(cZ)^2]^{k}}{[(\mathsf aX)^2+(\mathsf bY)^2]^n} \times\\
	\times   P_{2k,2n} \parentesis{\frac{\mathsf cZ}{\sqrt{(\mathsf aX)^2+(\mathsf bY)^2+(\mathsf cZ)^2}}}  \d X \d Y \d Z,
	\end{multline*}
	
	\noindent where $R$ is the mean radius of the ellipsoid, $\mathsf a$, $\mathsf b$ and $\mathsf c$ are its principal semi-axes, $\operatorname{Re}$ indicates the real part and $\mathsf B$ is the unit ball, defined by $X^2+Y^2+Z^2\le 1$. Moreover, $\mathcal Z_{2k,2n}$ can be written only in terms of $M$ and the principal moments of inertia because	
	\begin{equation*}
	\mathsf a=\sqrt{ \frac{5(- \mathcal A+ \mathcal B+ \mathcal C)}{2M}}, \quad 
	\mathsf b=\sqrt{ \frac{5( \mathcal A- \mathcal B+ \mathcal C)}{2M}}, \quad 
	\mathsf c=\sqrt{ \frac{5( \mathcal A+ \mathcal B- \mathcal C)}{2M}}.
	\end{equation*}
	
	Recalling the definitions of $q$ and $d$ in \cref{q_d}, the first non-vanishing Stokes coefficients are given by
	\begin{equation}\label{Z_spin-orbit}
	\mathcal Z_{0,0}=1,\qquad 
	\mathcal Z_{2,0} = -\frac{1}{2}\frac{q}{M R^2} , \qquad
	\mathcal Z_{2,2} = \sqrt{\frac{3}{8}} \, \frac{d}{M R^2},
	\end{equation}
	\begin{equation}\label{Z_spin-spin}
	\mathcal Z_{4,0}=\frac{15}{56} \frac{d^2 + 2q^2}{M^2 R^4} , \qquad
	\mathcal Z_{4,2}=-\frac{15}{28}\sqrt{\frac{5}{3}}\, \frac{d \, q}{M^2 R^4}  ,\qquad 
	\mathcal Z_{4,4}=\frac{15}{8} \sqrt{\frac{5}{14}} \, \frac{d^2}{M^2 R^4} ,
	\end{equation}
	
	\noindent and it seems that, in general, $\mathcal Z_{2k,2n}$ has the form of a homogeneous polynomial of degree $k$ with respect to $q/(MR^2)$ and $d/(MR^2)$.
\end{remark}

In order to simplify expression \eqref{V_expan_total}, recall that $\vec r$ is the vector pointing from the center of $\mathcal E_1$ to the center of $\mathcal E_2$. Then, the spherical coordinates of  $\vec r$ with respect to the inertial frame are $(r,\vartheta=\pi/2,\phi=f)$. The non-vanishing terms of \eqref{V_expan_total} are such that $\lambda_j=2l_j$ and $\mu_j=2m_j$. Let us call from now on $l=l_1+l_2$ and $m=m_1+m_2$. We can apply the formula
\begin{equation*}
Y_{2l,2m}(\pi/2,f)= \sqrt{\frac{(2l-2m)!}{(2l+2m)!}} P_{2l,2m} (0) e^{2im f},
\end{equation*}

\noindent and the following property of the associated Legendre polynomials
\begin{equation*}
P_{2l,2m}(0)=\frac{(-1)^{l-m}}{4^{l}} \frac{(2l+2m)!}{(l-m)!(l+m)!},
\end{equation*}

\noindent see for instance equation (68) in \cite{ste1973}. Then, we can write the potential keeping only the real part of $V$, so that the final expression potential is
\begin{equation}\label{V_full}
V=-\frac{G M_1 M_2 }{r} 
\sum_{\substack{(l_1,m_1)\in \Upsilon \\ (l_2,m_2)\in \Upsilon}}    \Gamma_{l_2,m_2}^{l_1,m_1}
\parentesis{\frac{R_1}{r}}^{2l_1} \parentesis{\frac{R_2}{r}}^{2l_2}  \mathcal Z^{1)}_{2l_1,2m_1} \mathcal Z^{2)}_{2l_2,2m_2}
\cos (2m_1(\theta_1-f)+2m_2 (\theta_2-f)) ,
\end{equation}

\noindent where 
\begin{equation}
\Gamma_{l_2,m_2}^{l_1,m_1} = 
\frac{(-1)^{l-m}}{4^{l}\sqrt{(2l_1-2m_1)!(2l_1+2m_1)!(2l_2-2m_2)!(2l_2+2m_2)!}}
\frac{(2l-2m)!(2l+2m)!}{(l-m)!(l+m)!}.
\end{equation}

The first terms of the expansion \cref{V_full} can be computed using \cref{Z_spin-orbit} and \cref{Z_spin-spin}. The terms corresponding to $l=l_1+l_2$, for $l=0,1$ and $2$, are shown in \cref{V_4_ph_parameters}.

\section*{Acknowledgements}

I would like to thank my PhD supervisors Prof. Alessandra Celletti and Prof. Rafael Ortega. A. Celletti guided me with the model, its derivation and applications. R. Ortega oriented me with the theoretical results and the overall approach. With their valuable suggestions, both of them encouraged me to develop this work and made me see its potential. I also thank Joan Gimeno for his help with the numerical plots in \Cref{fig:stab_several}.

This research was supported by the MSCA-ITN-ETN Stardust-R, Grant Agreement 813644.

%%%%%%%%%%%%%%%%%%%%%%%%%%%%%%%%%%%%%%%%%%%%%%% Bibliography goes here.

\bibliographystyle{siamplain}
\bibliography{references}

\end{document}